\documentclass[11pt]{amsart}
\usepackage{amsfonts,amsmath,amssymb}
\usepackage{mathabx}
\usepackage{mathrsfs}
\usepackage[latin1]{inputenc}
\usepackage[usenames,dvipsnames]{pstricks}
\newtheorem{theorem}{Theorem}[section]

\newtheorem{Corol}[theorem]{Corollary}

\newtheorem{Lemm}[theorem]{Lemma}

\newtheorem{proposition}[theorem]{Proposition}

\newcommand{\R}{\mathbb{R}}

\newcommand{\N}{\mathbb{N}}
\newcommand{\M}{{M}}

\newcommand{\g}{{g}}

\newcommand{\dd}{\mathrm{d}}
\newcommand{\dv}{\,\mathrm{dv}^{n}}

\newcommand{\dvv}{\,\mathrm{dv}^{2n-1}}
\newcommand{\ds}{\,\mathrm{d\sigma}^{n-1}}

\newcommand{\dss}{\,\mathrm{d\sigma}^{2n-2}}
\newcommand{\I}{\mathcal{I}}

\newcommand{\p}{\partial}

\newcommand{\norm}[1]{\|#1\|}
\newcommand{\abs}[1]{|#1|}
\newcommand{\set}[1]{\left\{#1\right\}}
\newcommand{\para}[1]{\left(#1\right)}
\newcommand{\cro}[1]{\left[#1\right]}

\newcommand{\seq}[1]{\left<#1\right>}

\newcommand{\To}{\longrightarrow}
\newcommand{\vv}{\mathrm{v}}
\newcommand{\hh}{\mathrm{h}}

\usepackage{graphicx}
\newcommand{\dive}{\textrm{div}}
\usepackage{mathrsfs}
\usepackage[latin1]{inputenc}
\usepackage[T1]{fontenc}
\usepackage{color,xspace}
\usepackage[usenames,dvipsnames]{pstricks}
\usepackage{times}
\allowdisplaybreaks
\begin{document}

\title[Determination of a vector field in a non-self-adjoint dynamical Schr\"odinger equation]{Stable determination of a vector field in a non-self-adjoint dynamical Schr\"odinger equation on Riemannian manifolds}
\author[M.~Bellassoued]{Mourad Bellassoued}
\author[I.~Ben A\"{\i}cha]{Ibtissem~Ben A\"{\i}cha}
\author[Z.~Rezig]{Zouhour Rezig}

\address{M.~Bellassoued. Universit\'e de Tunis El Manar, Ecole Nationale
d'ing\'enieurs de Tunis, ENIT-LAMSIN, B.P. 37, 1002 Tunis,
Tunisia}
\email{mourad.bellassoued@enit.utm.tn }
\address{I.~Ben A\"{\i}cha, Beijing Computational Science Research Center, Beijing 100193, China}
\email{ibtissem@csrc.ac.cn}
\address{Z.~Rezig, Universit\'e de Tunis El Manar, Facult\'e des Sciences de Tunis \& ENIT-LAMSIN, B.P. 37, 1002 Tunis, Tunisia}
\email{zouhourezig@yahoo.fr}
\date{\today}

\begin{abstract}
This paper deals with  an inverse problem for a non-self-adjoint Schr\"{o}dinger equation on a compact Riemannian manifold.  Our goal is to stably determine a real vector field from  the dynamical Dirichlet-to-Neumann map. We establish in dimension $n\geq2$, an  H\"{o}lder type  stability estimate  for the inverse problem under study. The proof is mainly based on the reduction to an equivalent problem for an electro-magnetic Schr\"{o}dinger equation and the use of  a  Carleman estimate  designed for elliptic operators.
\end{abstract}
\maketitle

\section{Introduction}\label{S1}
\setcounter{equation}{0}

 Let us consider a compact Riemannian manifold $(\M,g)$ of dimension $n\geq 2$. We denote by $\p \M$ its smooth boundary. Our goal is to determine a vector field in a non-self-adjoint Schr\"odinger equation.  We fix a coordinate system $x=\para{x^i}$ and let $\para{\frac{\p}{\p x^i}}$ be the corresponding tangent vector field.
We define the Laplace-Beltrami operator associated with the Riemannian metric $g$ as follows
\begin{equation}\label{1.1}
\Delta=\frac{1}{\sqrt{\abs{g}}}\sum_{j,k=1}^n\frac{\p}{\p x^j}\para{\sqrt{\abs{g}}\,g^{jk}\frac{\p}{\p x^k}}.
\end{equation}
In local coordinates we denote  $g=(g_{jk})$.  Here $(g^{jk})$ is the inverse of $g$ and $\abs{
g}=\det\,(g_{jk})$.  For any $x\in\M$, we define the inner product and the norm on the tangent space $T_x \M$ as follows
\begin{gather*}
g(X_1,X_2)=\seq{X_1,X_2}=\sum_{j,k=1}^n g_{jk}X_1^jX_2^k, \qquad X_j=\sum_{i=1}^nX_j^i\frac{\p}{\p x^i},\,\, j=1,2, \qquad
\abs{X}=\seq{X,X}^{1/2}.
\end{gather*}
 Let $T>0$, we set   $Q=\M\times(0,T)$ and $\Sigma=\p \M\times(0,T)$.  We define the anisotropic Sobolev space
\begin{equation*}
 \mathcal{H}^{2,1}(\Sigma)=H^{2}(0,T;L^{2}(\p \M))\cap L^{2}(0,T;H^{1}(\p\M)),
\end{equation*}
equipped with the norm
$\|\cdot\|_{\mathcal{H}^{2,1}(\Sigma)}=\|\cdot\|_{H^{2}(0,T;L^{2}(\p\M))}+\|\cdot\|_{L^{2}(0,T;H^{1}(\p\M))}.$ Let   $X\in\!\mathcal{C}^{\infty}(\M)$ \color{black} be a real vector field. We introduce the following initial boundary value problem  for the Schr\"odinger  equation
\begin{equation}\label{1.2}
\left\{
  \begin{array}{ll}
   \mathscr{L}_{X}u:=(i\p_{t} -\Delta +X) u=0 & \mbox{in} \,Q, \\
   \\
    u(\cdot, 0)=0 & \mbox{in}\, \M,  \\
    \\
    u=f & \mbox{on} \,\Sigma,
  \end{array}
\right.
\end{equation}
where  $f\!\in\! \mathcal{H}^{2,1}(\Sigma)$ and it satisfies  $f(\cdot,0)=\p_t f(\cdot,0) \equiv 0$.  Let $X$ be a vector field on $\M$, we define the vector field $\nabla u$ of a $\mathcal{C}^\infty$ function $u$ as follows
\begin{equation}\label{1.3}
X(u)=\seq{X,\nabla u}.
\end{equation}
In local  coordinates, we have
\begin{equation}\label{1.4}
\nabla u=\sum_{i,j=1}^n g^{ij}\frac{\p u}{\p x^i} \,\frac{\p}{\p x^j}.
\end{equation}
We denote by  $T_x^* \M$ the cotangent space which is the space of covectors or one-forms $Y=Y_j dx^j$. Here $(dx^j)$ is the basis of the cotangent space. We denote by $T \M$ (resp., $T^* M$) the tangent bundle  (resp., the cotangent bundle) of $ \M$ which is defined as the union  of the  spaces $T_xM$ (resp., the spaces $T^*_x M$), for any $x\in M$.
  We define the isomorphisms induced by $g$ as follows
  $$\begin{array}{ccc}
  \imath  : T_x \M& \longleftrightarrow& T^*_x \M\\
  X&\mapsto& X^\flat
  \end{array}, \qquad\qquad\qquad \begin{array}{ccc}
  \imath^{-1}  : T^*_x \M& \longleftrightarrow& T_x \M\\
 \,\,\,\, Y&\mapsto& Y^\sharp.
  \end{array} $$
In coordinates, the operators $\imath$ and $\imath^{-1}$ are defined by
\begin{equation}\label{1.9}
 X^\flat=X_jdx^j, \qquad Y^\sharp=Y^j\frac{\p}{\p x^j}, 
\end{equation}
with $X_j=\sum_{k=1}^ng_{jk}X^k$ and $ Y^j=\sum_{k=1}^n g^{jk}Y_k$. On the other hand, for any $Y_1, \,Y_2\in T^*_x M$, we define
\begin{equation}\label{1.10}
\seq{Y_1,Y_2}=\seq{Y_1^\sharp,Y_2^\sharp}=\sum_{j,k=1}^n g_{jk}Y_1^jY_2^k.
\end{equation}

In this paper, we aim to show that from the Dirichlet-to-Neumann (DN) map  $\Lambda_X: \mathcal{H}^{2,1}(\Sigma) \longrightarrow L^{2}(\Sigma)$, $
f\longmapsto \p_{\nu}u,
$ associated with  (\ref{1.2}),
 one can uniquely and stably determine the vector field $X$. Here $\nu= \nu(x)$ is the unit outward normal vector field  at $x \in \p \M$
  and $\p_{\nu}u$ stands for $\seq{\nabla u, \nu}$. In local coordinates,  $\nu$ and $\p_{\nu}u$ are given by
\begin{equation}\label{1.12}
 \nu=\sum_{j=1}^{n} \nu^j \frac{\p}{\p x_i}\quad \mbox{  and} \quad
\p_{\nu}u:=\seq{\nabla u, \nu}=\sum_{j,k=1}^{n}\g^{j,k}\nu_j \frac{\p u}{\p x^k},
\end{equation}
where  $\sum_{j,k=1}^{n}\g_{jk} \nu^j \nu^k=1$ and  $ \nu_j= \sum_{k=1}^{n}\g_{j,k} \nu^k.$
\smallskip

In this paper, we assume that the compact Riemannian manifold $(\M,\g)$ is \textit{simple}, i.e., is simply connected, any geodesic has no conjugate points and $\p\M$ is strictly convex. Any two points of the simple manifold $\M$ can be joined by a unique geodesic.
\smallskip

Namely, the main focus of this paper is the investigation
of the following problem:

\noindent \textbf{Problem 1:} Does a small perturbation on the flux measurement $\Lambda_X$ can cause an error in the determination of the vector $X$ on a simple compact  Riemannian manifold $\M$?

 The main idea in resolving this problem  is based on reducing it to an equivalent problem that we are familiar with and that it is easier to deal with. More precisely,  we will show that Problem 1 associated with (\ref{1.2}) can be equivalently reformulated to an other problem that concerns this equation
\begin{equation}\label{1.13}
\left\{
  \begin{array}{ll}
   \mathscr{H}_{A,q}u:=(i\p_{t}- \Delta_{A}+q)u=0 & \mbox{in} \,\,\,Q, \\
   \\
    u(\cdot,0)=0 & \mbox{in}\,\,\, \M,  \\
   \\
    u=f & \mbox{on} \,\Sigma,
  \end{array}
\right.
\end{equation}
with $f\in \mathcal{H}^{2,1}(\Sigma)$, $q:\M\mapsto \R$ is a real bounded electric potential and  $A=a_j dx^j\in \mathcal{C}^{\infty}(\M,T^* \M)$ \color{black} is a covector field with pure imaginary complex-valued coefficients $a_j\in\mathcal{C}^{\infty}(\M)$. Here $\Delta_{A}$   is given by
 \begin{equation}\label{1.14}
 \Delta_{A}=\Delta +2i \seq{A^{\sharp},\nabla }-i \delta\,A-\seq{A,A},
 \end{equation}
where  $A^\sharp$ is the vector field associated with the covector $A$ and $\delta$ is the coderivative  operator  defined by
\begin{equation}\label{1.15}
\delta A=\frac{1}{\sqrt{\abs{g}}}\sum_{j,k=1}^n \frac{\p}{\p x^j}\para{g^{jk}\sqrt{\abs{g}} a_k}.
\end{equation}
Note that the products   $\seq{A, A}$ and $\seq{A^\sharp,\nabla}$ are given by
\begin{equation}\label{1.16}
\seq{A,A}=\sum_{j,k=1}^n g^{jk}a_ja_k,\quad\mbox{and}\quad \seq{A^\sharp,\nabla}=\sum_{j,k=1}^n g^{jk}a_j\frac{\p}{\p x^k}.
\end{equation}
Keeping the above points in mind, the idea is then to move from the problem of determining $X$ appearing in (\ref{1.2}) from the DN map $\Lambda_X$ to the problem of determining $A$ and $q$ appearing in (\ref{1.13}) from the equivalent DN map $N_{A,q}:f\longmapsto (\p_{\nu}+i\seq{A^\sharp, \nu})u$\,\, associated with the equation (\ref{1.13}).
 It should be noticed that there is an obstruction in determining A from $N_{A,q}$ since  $N_{A,q}$ is invariant under the gauge transformation (see \cite{[Sun]} for more information). It is well known that for any $A \in H^k(\M,T^*\M)$, there exists a unique
$A^s \in H^k(\M,T^*\M)$ and $\varphi \in H^{k+1}(\M)$ such that:
\begin{equation}\label{1.18}
A=A^s+d\varphi,\quad \delta A^s=0,\quad \varphi|_{\p\M}=0,\quad \mbox{and}\quad  d\varphi=\sum_{j=1}^n\frac{\p\varphi}{\p x^j}dx^j.
\end{equation}
Here $A^s$ is said the solenoidal part of $A$ and $d\varphi$ is its potential part. The best we could hope to determine from $N_{A,q}$ is the solenoidal part $A^s$ of the covector $A$. In order to deal with our main problem we first need to deal with this equivalent problem:

\noindent\textbf{Problem 2:} Is it possible to stably recover  the electric potential $q$ and the solenoidal part $A^s$ of the covector $A$ defined on a simple compact Riemannian manifold from the knowledge of the DN map $N_{A,q}$  under certain conditions?

 Actually, Problem 2 is closely related to the one considered by Bellassoued \cite{[Bellass]} in the case where the covector field $A$ is with real valued coefficients. But here we formulate the problem for complex  vector fields. Theorem \ref{Thm2.3} answers this problem affirmatively.

 The uniqueness in recovering terms in  Riemannian  non-self-adjoint operators, was recently considered by Krupchyk and Uhlmann in \cite{[KU]}. They proved the unique identifiability for an advection term from the knowledge of the DN map measured on the boundary of the manifold. We can also refer to the paper of Kurylev and Lassas \cite{[BA5]} in which a uniqueness result for a general non-self-adjoint second-order elliptic operator on a manifold with boundary is addressed.

In contrast to the Riemannian case,  the problem of recovering coefficients in non-self-adjoint operators has been extensively  studied in the euclidian case. We cite for example the paper of  Pohjola \cite{[BIB18]}, in which  an inverse problem for the recovery of a velocity field in a steady state convection diffusion equation was considered. A  uniqueness result for this problem has been proven and the proof was mainly based on  reducing it to an auxiliary problem for a stationary magnetic Schr\"odinger equation. Cheng, Nakamura and Somersalo \cite{[BIB11]} studied the same problem and proved a uniqueness result but for more regular coefficients. Salo \cite{[BIB20]} also treated the uniqueness issue for the recovery of Lipschitz continuous coefficients.

 As for stability results in the euclidian case, we cite the work of Bellassoued and Choulli \cite{[Bel-Choul]} where they proved in dimension $n \geq 2$ that the knowledge of the DN map for the magnetic Schr\"odinger equation measured on the boundary of a bounded smooth domain of $\R^n $ determines uniquely the magnetic field. They also proved a H\"older-type stability in determining the magnetic field introduced by the magnetic potential.
 We also cite the work of Bellassoued and Ben A\"icha \cite{[BMBI]}, in which they focused on the study of an inverse problem for a non-self-adjoint hyperbolic equation and they proved a stability of H\"older type in recovering a first order coefficient appearing in a wave equation from the knowledge of Neumann boundary data.  The overall  idea in resolving these problems is based on bringing the problems under investigation back to similar ones that we are familiar with.

  In this paper, our objective is the study of the inverse problem associated with the equation (\ref{1.13}).  Inspired by the work of  Bellassoued and Rezig \cite{[BZ]}, Bellassoued and Ben A\"icha \cite{[BMBI]} and the paper of Bellassoued \cite{[Bellass]}, we aim to stably recover the vector field $X$ from the DN map $\Lambda_X$ and  show a stability of H\"older type.
   It seems that the present paper is the first proving a stability result for a Riemannian non-self-adjoint operator.

  The remainder of this paper is organized as follows:  in Section \ref{S2}, we state the main results answering to both problems  and prepare  the necessities to prove these statements. In Section \ref{S3}, we process the geodesic ray transformation for one-forms and functions on a manifold. Section \ref{S4} is devoted to the study of the preliminary problem (Problem 2). In Section \ref{S5}, we deal with the main problem of this paper (Problem 1) and by the use of an appropriate Carleman estimate, we establish  a stability estimate for the recovery of the real vector field $X$ .

\section{Preliminaries and main results}\label{S2}
\setcounter{equation}{0}

In this section we state the main results answering Problems 1 and 2. Let us first set up some notations and terminologies that will be used in this rest of the paper.
We denote by $\dv=|g|^{1/2}d x_1\wedge\cdots \wedge  d x_n$ the Riemannian volume induced by $g$.  Let  $L^2(\M)$ be the completion
of $\mathcal{C}^\infty(\M)$ with the inner product
\begin{equation}\label{2.1}
\para{u_1,u_2}=\int_\M u_1(x) \overline{u_2(x)} \dv,\qquad  u_1,u_2\in\mathcal{C}^\infty(\M).
\end{equation}
 Let us denote by $\mathcal{C}^\infty(\M,T\M)$ the space of smooth vector fields and by $\mathcal{C}^\infty(\M,T^*\M)$ the space of smooth one forms. On the other hand, we define $L^2(\M,T^*\M)$ and  $L^2(\M,T\M)$ by the inner product
\begin{equation}\label{2.2}
\para{X,Y}=\int_\M\seq{X,\overline{Y}}\dv,\quad X,Y \in L^{2}(\M).
\end{equation}
We denote by $H^k(\M)$ the Sobolev space endowed with the norm
\begin{equation}\label{2.3}
\norm{u}^2_{H^k(\M)}=\norm{u}^2_{L^2(\M)}+\sum_{k=1}^n\|\nabla^k u\|^2_{L^2(\M,T^k \M)}.
\end{equation}
Here $\nabla^k$ denotes the covariant differential of $u$. For any one form $A=a_j dx^j$ in $H^k(\M,T^*\M)$ we denote
\begin{equation}\label{2.4}
\norm{A}_{H^k(\M,T^*\M)}=\sum_{j=1}^n\norm{a_j}_{H^k(\M)}.
\end{equation}
Before stating our main result let us introduce the admissible set of the unknown vectors $X$. Given $m_1 \geq 0$ and $k> n/2+2$ \color{black}. We define the following set
$$\mathscr{X}(m_1):=\{X\in \mathcal{W}^{2,\infty}(\M,T\M),\,\,\|X\|_{H^k(\M,T\M)}\leq m_1\}.$$

Let $x \in \M$ and let $\pi$ be a two-dimensional subspace of $ T_x\M$ spanned by $\eta$ and $\xi$. The number
$$
K(x,\pi)=\frac{\langle R(\xi,\eta)\eta,\xi\rangle}{\vert \xi \vert^2 \vert \eta \vert^2-\langle \xi, \eta \rangle^2},
$$
is independent of the choice of $\xi$ and $\eta$. It is the sectional curvature of the manifold $\M$ at the point $x$ and in the two-dimensional direction $\pi$.\\
For $(x,\xi)\in T\M$, we set
\begin{equation}\label{3.8}
 K(x,\xi)=\underset{\pi \ni \xi}{\sup} K(x,\pi) \quad \mbox{ and } \quad  K^+(x,\xi)=\max\{0,K(x,\xi)\}.
\end{equation}
 If the compact Riemannian manifold $(\M,\g)$ is simple, we define
 \begin{equation}\label{3.9}
 k^+(\M,\g)=\sup\lbrace \int_0^{\tau_+(x,\xi)} tK^+(\gamma_{x,\xi}(t),\dot{\gamma}_{x,\xi}(t)) dt, \; (x,\xi)\in \p_+S\M\rbrace.
 \end{equation}

Then, our main result can be stated as follows
\begin{theorem}\label{Thm2.1}
Let $(M,g)$ be a simple compact Riemannian manifold with boundary of dimension $n\geq2$ such that $k^+(M,g)<1/2$. Let $m_1\geq0$, $T>0$.  There exist positive constants $C$ and $\tilde{s}>0$ such that
\begin{equation}\label{2.5}
\|X_1-X_2\|_{L^2(\M,T\M)}\leq C\|\Lambda_{X_1}-\Lambda_{X_2}\|^{\tilde{s}},
\end{equation}
for any $X_1, \,X_2\in \mathscr{X}(m_1)$ such that $X_1=X_2$ on $\p\M$. Here the constant $C$ is depending only on $\M$ and the norm $\|\cdot \|$ denotes the norm in $\mathcal {L}(\mathcal{H}^{2,1}(\Sigma),L^{2}(\Sigma)).$
\end{theorem}
To prove Theorem \ref{Thm2.1} we need to reduce the problem associated with (\ref{1.2}) to an equivalent problem that concerns  the electro-magnetic equation (\ref{1.13}). So that  determining $X$ appearing in (\ref{1.2}) from $\Lambda_X$ will amount to determining $A^s$ and $q$ in (\ref{1.13}) from $N_{A,q}.$

 \noindent Let  $\dive X$  denote the divergence of a vector field $X\in H^1(\M,T \M)$ on $\M$. In  coordinates, we have \begin{equation}\label{2.6}
\dive X=\frac{1}{\sqrt{\abs{g}}}\sum_{i=1}^n \frac{\p}{\p x^i}\para{\sqrt{\abs{g}}\,X^i},\quad X=\sum_{i=1}^n X^i\frac{\p}{\p x^i}.
\end{equation}
The coderivative operator $\delta$  can be seen as the adjoint of the exterior derivative $-d$ as follows
\begin{equation}\label{2.7}
\para{\delta A,v}=-\para{A,dv},\quad A\in\mathcal{C}^\infty(\M,T^*\M),\,\,\,\,v\in \mathcal{C}^\infty(\M),
\end{equation}
such that $A_{|\partial M}=0$. For any $X\in H^1(\M,T \M)$, the divergence formula is given by
\begin{equation}\label{2.8}
\int_\M \dive X \dv=\int_{\p \M}\seq{X,\nu} d\sigma^{n-1},
\end{equation}
where $\ds$ is the volume form of $\p \M$.   On the other hand,  for any function $u\in H^1(\M)$ we have
\begin{equation}\label{2.9}
\int_\M \dive X\,u\,\dv=-\int_\M\seq{X,\nabla u} \,\dv+\int_{\p M}\seq{X,\nu} u \,d\sigma^{n-1}.
\end{equation}
Thus if $u\in H^1(\M)$ and $w\in H^2(\M)$, the following identities hold
\begin{equation}\label{2.10}
\int_\M\Delta w u\,\dv=-\int_\M\seq{\nabla w,\nabla u} \,\dv+\int_{\p \M}\p_\nu w u \,d\sigma^{n-1},
\end{equation}
and
\begin{equation}\label{2.11}
\int_\M \Delta w u\dv=\int_\M \Delta u\, w\dv+\int_{\p \M} (\p_{\nu} w u-\p_{\nu}u w)\, d\sigma^{n-1}.
\end{equation}
Let us introduce the following set
$\mathcal{H}^{2,1}_{T}(\Sigma):=\!\{g\in\mathcal{H}^{2,1}(\Sigma),\,\,g(\cdot,T)=0\}.$
 For any $g \in \mathcal{H}^{2,1}(\Sigma)$, we introduce  the adjoint operator of the DN map $\Lambda_{X}$ as follows:
$$\begin{array}{ccc}
\Lambda^{*}_{X}: \mathcal{H}_{T}^{2,1}(\Sigma)&\longrightarrow& L^{2}(\Sigma)\\
g&\longmapsto &\p_{\nu}v,
\end{array}$$
where $v$ here is the unique solution to this equation
\begin{equation}\label{2.13}
\left\{
  \begin{array}{ll}
   \mathscr{L}^{*}_{X}v=(i\p_{t}-\Delta -X-\mbox{div}X)v=0 & \mbox{in} \,\,\,Q, \\
  \\
    v(\cdot,T)=0 & \mbox{in}\,\,\, \M,  \\
   \\
    v=g& \mbox{on} \,\,\,\Sigma.
  \end{array}
\right.
\end{equation}
Next, we denote
$$\begin{array}{ccc}
N^{*}_{A,q}: \mathcal{H}_{T}^{2,1}(\Sigma)&\longrightarrow& L^{2}(\Sigma)\\
g&\longmapsto &(\p_{\nu}-i \seq{A^\sharp,\nu} ) v,
\end{array}$$
associated with this problem
\begin{equation}\label{2.14}
\left\{
  \begin{array}{ll}
  \mathscr{H}^{*}_{A,q}v=\mathscr{H}_{-A,q}v=0 & \mbox{in} \,\,\,Q, \\
 \\
    v(\cdot,T)=0 & \mbox{in}\, \M,  \\
   \\
    v=g& \mbox{on} \,\Sigma.
  \end{array}
\right.
\end{equation}
We should notice that  $ N_{-A,q}=N_{A,q}^{*}$.
We aim now to choose specific  $A$ and $q$  in such a way $\mathscr{H}_{A,q}$ co\"incides with $\mathscr{L}_{X}$  and the same for the corresponding DN maps $\Lambda_{X}$ and $N_{A,q}.$  Let us  state  a lemma that will play an important role in showing Theorem \ref{Thm2.1}.
\begin{Lemm}\label{Lm2.2}
 For $j=1,2$, let $X_j\in\mathscr{X}(m_1)$ and $X_j^{\flat}$  its associated covector. We define  $A_j$ and $q_j$ as follows\color{black}
 \begin{equation}\label{2.15}
  A_{j}=\frac{i}{2}X^{\flat}_{j},\quad
 \mbox{and} \quad q_{j}=\frac{1}{4}\seq{X_j,X_j}-\frac{1}{2}\,\dive X_{j},\quad j=1,\,2.
 \end{equation}
 Then,  the following identities hold
 \begin{equation*}
 \mathscr{H}_{A_{j},q_{j}}=\mathscr{L}_{X_{j}},\quad \quad \mathscr{H}^{*}_{A_{j},q_{j}}=\mathscr{L}^{*}_{X_{j}}\qquad \mbox{and}\qquad
\|N_{A_{1},q_{1}}-N_{A_{2},q_{2}}\|=\|\Lambda_{X_{1}}-\Lambda_{X_{2}}\|.
\end{equation*}
Here $\|\cdot\|$ denotes the norm in the space $\mathcal{L}(\mathcal{H}^{2,1}(\Sigma);L^{2}(\Sigma))$.
 \end{Lemm}
 \begin{proof}
 From (\ref{2.15}) and using the fact that div $X=\delta X^{\flat}$, one can check that
 \begin{equation}\label{2.16}
\mathscr{H}_{A_{j},q_{j}}u=(i\p_{t}-\Delta _{A_{j}}+q_{j}(x))u=i\p_{t}u-\Delta u + \seq{X,\nabla u} =(i\p_{t}-\Delta + X)u=\mathscr{L}_{X_{j}}u,
 \end{equation}
 and
 \begin{equation}\label{2.17}
 \mathscr{H}^{*}_{A_{j},q_{j}}v=(i\p_{t}-\Delta_{(-A_{j})}+q_{j}(x))v=(i\p_{t}-\Delta -X-\dive X) v= \mathscr{L}^{*}_{X_{j}}v.
 \end{equation}
In order to prove the last identity, we consider  by $u_{j}$ and $v_{j}$ for   $j=1,\,2$,  two solutions of
\begin{equation}\label{2.18}
\left\{
  \begin{array}{ll}
  \mathscr{L}_{X_{j}}u_{j}=0 & \mbox{in} \,\,\,Q, \\
  \\
    u_{j}(\cdot,0)=0 & \mbox{in}\,\,\, \M,  \\
   \\
    u_{j}=f & \mbox{on} \,\,\,\Sigma,
  \end{array}
\right. ;\qquad \qquad \left\{
  \begin{array}{ll}
  \mathscr{L}^{*}_{X_{j}}v_{j}=0 & \mbox{in} \,\,\,Q, \\
  \\
    v_{j}(\cdot,T)=0 & \mbox{in}\,\,\, \M,  \\
  \\
    v_{j}=g & \mbox{on} \,\,\,\Sigma,
  \end{array}
\right.
\end{equation}
with $f\in\mathcal{H}^{2,1}(\Sigma)$ and $g\in \mathcal{H}^{2,1}_{T}(\Sigma)$.
We  multiply the first equation in the left hand side of (\ref{2.18}) by $\overline{v}_{j}$ and we integrate by parts, we obtain
\begin{equation}\label{2.19}
 \int_{0}^T\int_{\p\M}\Lambda_{X_{j}}(f) \overline{g}\,\ds\,dt=\int_{0}^T\int_{\M} (i\p_{t}u_{j}\overline{v}_{j}+ \seq{\nabla u_{j}, \nabla \overline{v}_{j}}+ X_{j}( u_{j})\overline{v}_{j})\,\dv\,dt.
 \end{equation}
From (\ref{2.16}) and (\ref{2.17}), the solutions  $u_{j}$  and $v_{j}$ with  $j=1,\,2$, also solve
 \begin{equation}\label{2.20}
\left\{
  \begin{array}{ll}
  \mathscr{H}_{A_{j},q_{j}}u_{j}=0 & \mbox{in} \,\,\,Q, \\
  \\
    u_{j}(\cdot,0)=0 & \mbox{in}\,\,\, \M,  \\
   \\
    u_{j}=f & \mbox{on} \,\,\,\Sigma,
  \end{array}
\right.; \qquad\qquad \left\{
  \begin{array}{ll}
   \mathscr{H}^{*}_{A_{j},q_{j}}v_{j} =0 & \mbox{in} \,\,\,Q, \\
  \\
    v_{j}(\cdot,T)=0 & \mbox{in}\,\,\, \M,  \\
   \\
    v_{j}=g & \mbox{on} \,\,\,\Sigma.
  \end{array}
\right.
\end{equation}
On the other hand, if we multiply the equation in the left hand side of (\ref{2.20}) by $\overline{v}_{j}$ and we integrate by parts, we obtain
\begin{eqnarray}\label{2.21}
\int_{0}^T\int_{\p\M} N_{A_{j},q_{j}}(f) \overline{g}\,\ds\,dt=\int_{0}^T\int_{\M} (i\p_{t}u_{j}\overline{v}_{j}+\seq{\nabla u_{j}, \nabla \overline{v}_{j}}+X_{j}(u_{j})\overline{v}_{j}) \,\dv\,dt\cr-\frac{1}{2}\int_{0}^T\int_{\p\M}\seq{V_{j}, \nu} \,u_{j}\overline{v}_{j}\,\ds\,dt.
\end{eqnarray}
Therefore, in light of (\ref{2.19}) and (\ref{2.21}), one gets
$$\int_{0}^T\int_{\p\M} N_{A_{j},q_{j}}(f) \overline{g}\,\ds\,dt=\int_{0}^T\int_{\p\M}\Lambda_{X_{j}}(f) \overline{g}\,\ds\,dt-\frac{1}{2}\int_{0}^T\int_{\p\M}\seq{X_{j}, \nu}\,f\,\overline{g}\,\ds\,dt.$$
Thus, using the fact that  $\seq{X_{1}, \nu}=\seq{X_{2}, \nu}$ on $\p\M$, \color{black} we get the desired result.
 \end{proof}

 Thanks to Lemma \ref{Lm2.2},  the problem under study  is equivalently reformulated as to whether the solenoidal part $A^s$ of the magnetic potential  and the electric potential $q$ in (\ref{1.13})  can be retrieved or not  from   $N_{A,q}$. This will be the goal of Section \ref{S4}.

 We move now to introduce the admissible sets. Let $m_1,m_2>0$ and $k>n/2+2$ be given \color{black}, we define
\begin{equation}\label{2.22}
\mathscr{A}(m_1,k)= \set{A\in W^{2,\infty}(\M,T^*\M),\,\,\norm{A}_{H^k(\M,T^*\M)}\leq m_1},
\end{equation}
and
\begin{equation}\label{2.23}
\mathscr{Q}(m_2)= \set{q\in W^{1,\infty}(\M),\,\,\norm{q}_{W^{1,\infty}(\M)}\leq m_2}.
\end{equation}
\begin{theorem}\label{Thm2.3}
 Let $(M,g)$ be a simple compact Riemannian manifold with boundary of dimension $n\geq2$ such that $k^+(M,g)<1/2$. There exist  $C>0$ and $\kappa\in(0,1)$ such that for any $A_1,A_2\in\mathscr{A}(m_1,k)$ and $q_1,q_2\in\mathscr{Q}(m_2)$  such that they coincide on the boundary $\p \M$, the following estimate holds true
\begin{equation}\label{2.24}
\norm{A_1^s-A_2^s}_{L^2(\M,T^*\M)}+\norm{q_1-q_2}_{L^2(\M)}\leq C\norm{N_{A_1,q_1}-N_{A_2,q_2}}^\kappa
\end{equation}
where $C$ depends on $\M$, $m_1,m_2$ and $n$.
\end{theorem}

\section{Geodesical X-ray tranforsm on a simple Manifold}\label{S3}
\setcounter{equation}{0}
In this section we consider simple manifolds and we deal with geodesic $X$-ray transform of a function or a covector field. Our aim is to state a stability result. We have such results proved on simple surface by Mukhometov \cite{[Mukh]}. Concerning simple manifolds of any dimension, we find stability estimates in \cite{[SU2]}, \cite{[SU3]}, and also in V. A. Sharafutdinov's book \cite{[Sh]}. This result has also been generalized to nontrapping manifolds without conjugate points by Dairbekov in \cite{[Dai]}.

\subsection{Inverse inequality for geodesic ray transform of a function on a simple manifold}
We start by describing the environment where we work.

We consider a compact Riemannian manifold $(\M,\, \g)$ with boundary. We say that $(\M,\g)$ is a \emph{ convex non-trapping manifold} if the boundary $\p \M$ is strictly convex (that means the second fundamental form of the boundary is positive definite at every
    boundary point) and if all the geodesics have finite length in $\M$.\\
    A compact convex non-trapping manifold is said to be \emph{simple} if we have no conjugate points on any geodesic.\\
     We cite the main properties of a simple manifold that will be  used in this paper: a simple Riemannian manifold of dimension $n$ is diffeomorphic to a closed ball in $\R^n$, and for any two points in the manifold there exists an unique geodesic joining them.
     \\
      Let us define the sphere bundle and the co-sphere bundle of $\M$ by \begin{align*}
S\M=\set{(x,\xi)\in T\M;\,\abs{\xi}=1} \quad \mbox{and}\quad
S^*\M=\set{(x,p)\in T^*\M;\,\abs{p}=1}.
\end{align*}
     For $x\in \M$ and $\xi\in T_x\M$, we let $\gamma_{x,\xi}$ the unique geodesic starting from $x$ in the direction $\xi$; that means $\gamma_{x,\xi}(0) = x$ and $\dot{\gamma}_{x,\xi}(0) = \xi$.  We define the exponential map $\exp_x:T_x\M\to \M$  by
\begin{equation}\label{3.1}
\exp_x(v)=\gamma_{x,\xi}(\abs{v}),\quad \xi=\frac{v\,\,}{\abs{v}}.
\end{equation}
In the sequel, we suppose that the manifold $(\M,\g)$ is simple. Then the map $\exp_x$ is a global diffeomorphism.

For $(x,\xi)\in S\M$, we have an unique geodesic $\gamma_{x,\xi}$ corresponding to $(x,\xi)$ defined on a maximal finite interval $[\tau_-(x,\xi),\tau_+(x,\xi)]$, such that $\gamma_{x,\xi}(\tau_\pm(x,\xi))\in\p\M$. The corresponding geodesic flow is defined by
$$\phi_t:S\M\to S\M,$$
\begin{equation}\label{3.2}
\phi_t(x,\xi)=(\gamma_{x,\xi}(t),\dot{\gamma}_{x,\xi}(t)),\quad t\in [\tau_-(x,\xi),\tau_+(x,\xi)].
\end{equation}
We obviously have $\phi_t\circ\phi_s=\phi_{t+s}$.  We define the vector field $H$  associated with the geodesic flow $\phi_t$ by setting, for $u\in\mathcal{C}^\infty(S\M)$ and $(x,\xi)\in S\M,$
 \begin{equation}\label{Hu}
 Hu(x,\xi)=\frac{d}{dt}u(\phi_t(x,\xi))_{|t=0}.
\end{equation}

\smallskip

Now, we split the boundary of the manifold $S\M$ in two compact submanifolds of inner and outer vectors. We set
\begin{equation}\label{3.3}
\p_{\pm}S\M =\set{(x,\xi)\in S\M,\, x \in \p \M,\, \pm\seq{\xi,\nu(x)}< 0},
\end{equation}
where $\nu$ is the unit outer normal to the boundary. The manifolds  $\p_+ S\M$ and $\p_-S\M$ have the same boundary $S(\p \M)$, and we have $\p S\M = \p_+ S\M\cup \p_- S\M$.
   Let $\mathcal{C}^\infty(\p_+ S\M)$ be the space of smooth functions on the manifold $\p_+S\M$ and define the functions $\tau_\pm:S\M\to\R$  as in (\ref{3.2}). We have the following properties:
$$
\tau_-(x,\xi)\leq 0,\quad \tau_-(x,\xi)=0,\quad (x,\xi)\in\p_+S\M,\quad \tau_-(\phi_t(x,\xi))=\tau_-(x,\xi)-t,
$$
$$\tau_+(x,\xi)\geq 0,\quad
\tau_+(x,\xi)=0,\quad (x,\xi)\in\p_-S\M,\quad \tau_+(\phi_t(x,\xi))=\tau_+(x,\xi)-t,
$$ and
$$ \tau_+(x,\xi)=-\tau_-(x,-\xi).$$
In particular if $(x,\xi)\in\p_+ S\M$, the maximal
geodesic $\gamma_{x,\xi}$ satisfying the initial conditions $\gamma_{x,\xi}(0) = x$ and $\dot{\gamma}_{x,\xi}(0) = \xi$ is defined on  $[0,\tau_+(x,\xi)]$.
\smallskip

The functions $\tau_{\pm}(x,\xi)$  are smooth near a point $(x,\xi)$ whose geodesic $\gamma_{x,\xi}(t)$ intersects the boundary $\p\M$ transversely for $t=\tau_{\pm}(x,\xi)$.
Some derivatives of $\tau_{\pm}(x,\xi)$ are unbounded in a neighbourhood of any point of $T\M \cap T(\p\M)$. So such points are singular and the strict convexity of $\p\M$ implies that $\tau_{\pm}(x,\xi)$ are smooth on $T\M \setminus T(\p\M).$ In particular, $\tau_{+}$ is smooth on $\p_+SM$.
\smallskip

Let $$
\dd \omega_x(\xi)=\sqrt{\abs{\g}} \, \sum_{k=1}^n(-1)^k\xi^k \dd \xi^1\wedge\cdots\wedge \widehat{\dd \xi^k}\wedge\cdots\wedge \dd \xi^n,
$$ be
the volume form defined on $S_x\M$, induced by the Riemannian scalar product on $T_x\M$. Here, the notation $\, \widehat{\cdot} \,$ means that the corresponding factor is to be omitted. And let
$$
\dvv (x,\xi)=\dd\omega_x(\xi)\wedge \dv,
$$ be the volume form $\dvv$ on the manifold $S\M$ where $\dv$ denote is the Riemannian volume form on $\M$. By Liouville's theorem, the geodesic flow preserves the volume form $\dvv$. Thus, if  $\ds$ denotes the volume form of $\p \M$ then we define the volume form on the boundary $\p S\M =\set{(x,\xi)\in S\M,\, x\in\p \M}$ by
$$
\dss=\dd\omega_x(\xi) \wedge \ds.
$$
We denote by $R$ the curvature tensor of the Levi-Civita connection $\nabla_X$ given by
$$
R(X,Y)Z=\nabla_X\nabla_Y Z+\nabla_Y\nabla_X Z-\nabla_{[X,Y]} Z.
$$

Now let $L^2_\mu(\p_+S\M)$ be the Hilbert space of square integrable functions with respect to the measure $\mu(x,\xi)\dss$ with
$\mu(x,\xi)=\abs{\seq{\xi,\nu(x)}}$ equipped with the scalar product
\begin{equation}\label{3.4}
\para{u,v}_\mu=\int_{\p_+S\M}u(x,\xi) \overline{v}(x,\xi) \mu(x,\xi)\dss.
\end{equation}
We define the geodesic $X$-ray transform on the manifold $\M$ by the operator
\begin{equation}\label{3.5}
\I:\mathcal{C}^\infty(\M)\To \mathcal{C}^\infty(\p_+S\M),
\end{equation}
defined by
\begin{equation}\label{3.6}
\I f(x,\xi)=\int_0^{\tau_+(x,\xi)}f(\gamma_{x,\xi}(t))\, \dd t,\quad (x,\xi)\in\p_+S\M.
\end{equation}
Since $\tau_+(x,\xi)$ is a smooth function on $\p_+S\M$ (see Lemma 4.1.1 of \cite{[Sh]}) then $\I f$ is a smooth function on $\p_+S\M$. Thus, for every integer $k\geq 1$, we can extend $\I$ as a bounded operator
\begin{equation}\label{3.7}
\I:H^k(\M)\To H^k(\p_+S\M).
\end{equation}
 \smallskip
The following stability's result for the $X$-ray transform of functions will be crucial in the proof of the main theorem \ref{Thm2.1} of this paper. We can find its proof in \cite{[BZ]}.
\begin{theorem}\label{Thm3.1}
Let $(\M,\g)$ be a simple compact Riemannian manifold  with $k^+(\M,\g)<1$, then there exist a constant $C>0$ such that the stability estimate
\begin{equation}\label{3.10}
 \Vert f \Vert _{L^2(\M)}\leq C \Vert \I f \Vert_{H^1(\p_+S\M)}
\end{equation}
holds true for any $f\in H^1(\M)$.
\end{theorem}

\subsection{Inverse inequality for geodesic $X$-ray transform of 1-forms on a simple manifold}
In this subsection, we define the geodesic $X$-ray transform of a 1-form on a simple Riemannian manifold $(\M,\g)$ as being the linear operator:
\begin{equation}\label{3.49}
\I_1:\mathcal{C}^\infty(\M,T^*\M)\To \mathcal{C}^\infty(\p_+S\M),
\end{equation}
defined by the equality
\begin{equation}\label{3.50}
\I_1 (A)(x,\xi)= \int_{\gamma_{x,\xi}} A= \, \sum_{j=1}^n \, \int_0^{\tau+(x,\xi)}a_j(\gamma_{x,\xi}(t)) \dot{\gamma}^j_{x,\xi}(t)\,  \dd t,\quad (x,\xi)\in\p_+S\M,
\end{equation}
where $\gamma_{x,\xi} : [0,\tau_+(x,\xi)] \to \M$ is the maximal
geodesic starting at $x$ with initial velocity $\xi$.  We have obviously $\I_1(d\varphi)=0$ for any smooth function $\varphi$ on $\M$ satisfying the condition $\varphi_{| \p\M}=0.$\\
Like for the ray transform of functions defined above, we extend
the ray transform $\I_1$ on a simple manifold as a bounded operator
\begin{equation}\label{3.51}
\I_1:H^k(\M,T^*\M)\To H^k(\p_+S\M), \, k \geq 1.
\end{equation}

For every magnetic potential $A$, we have the following decomposition (see Theorem 3.3.2 p89 in \cite{[Sh]}).
\begin{Lemm}\label{Lm3.2}
Let $\M$ be a compact Riemannian manifold with boundary and let $k\geq 1$ be an integer. For every covector field $A \in H^k(\M,T^* \M),$ there exist uniquely determined $A^s \in H^k(\M,T^* \M)$ and $\varphi \in H^{k+1}(\M)$ such that
\begin{equation}\label{3.52}
A= A^s+d\varphi, \quad \quad \delta A^s=0, \quad \quad \varphi_{| \p\M}=0.
\end{equation}

Furthermore, we have
\begin{equation}\label{3.53}
\Vert A^s \Vert_{H^k(\M,T^* \M)} \leq C \Vert A \Vert_{H^k(\M,T^* \M)}, \quad \quad \Vert \varphi \Vert_{H^{k+1}(\M)} \leq C \Vert \delta A \Vert_{H^{k-1}(\M)}.
 \end{equation}
 The constant $C$ is independent of $A$. In particular,  $A^s$ and $\varphi$ are smooth if $A$ is smooth.
\end{Lemm}
  If $(\M,\g)$ is a simple manifold, it is known that $\I_1$ is injective on the set of solenoidal 1-forms. We emphasize that by definition of $\I_1$ and by the boundlessness of the trace operator, we have the following lemma.
  \begin{Lemm}\label{Lm3.3}
   If $(\M,\g)$ is a simple manifold and if $\varphi \in H^{k+1}(\M)$  ($ k \geq 1$)  satisfies the boundary condition $\varphi_{| \p\M}=0,$ then $\I_1(d\varphi)=0.$
\end{Lemm}
Consequently, the best we could hope to recover from the ray transform, is the solenoidal part $A^s$ of the covector $A$.


The main result of this subsection is the following theorem. We recall that $\nu$ is
 the unit outer normal to the boundary and that $k^+$ is defined by (\ref{3.9}).

 \begin{theorem}\label{Thm3.4}
 Let $(\M,\g)$ be a simple manifold with $k^+(\M,\g)<\frac{1}{2}$. Then for every
  covector field $A \in H^k(\M,T^* \M),$ the stability estimate
 \begin{align}\label{3.54}
\Vert A^s \Vert^2_{L^2(\M,T^* \M)} & \leq C \big( \Vert \langle \nu,(A^s) ^\sharp_{| \p\M}
 \rangle \Vert_{L^2(\p\M,T \M)}\Vert \I_1 (A) \Vert_{L^2(\p_+S\M)}+ \Vert \I_1 (A) \Vert^2_{H^1(\p_+S\M)}\big),
\end{align}
 holds true. The constant $C$ is independent of $A$.
 \end{theorem}

Using the estimate of Lemma \ref{Lm3.2}, we deduce the following result.
 \begin{Corol}\label{Corol3.5}
  Let $(\M,\g)$ be a simple manifold with $k^+(\M,\g)<\frac{1}{2}$. Then for every  covector field $A \in H^1(\M,T^* \M),$ the following stability estimate
 \begin{align}\label{3.55}
\Vert A^s \Vert^2_{L^2(\M,T^* \M)} &
\leq C_1 \big( \Vert A \Vert_{H^1(\M,T^* \M)}\Vert \I_1 (A) \Vert_{L^2(\p_+S\M)}+
 \Vert \I_1 (A) \Vert^2_{H^1(\p_+S\M)}\big)
\end{align}
 holds true. The constant $C_1$ is independent of $f$.

 \end{Corol}
 With respect to Lemma \ref{Lm3.3}, it suffices to prove the Theorem  \ref{Thm3.4} for  $A \in H^k(\M,T^* \M)$ satisfying  $\delta A=0$.
 
By using density arguments, it's enough  to prove the theorem for a real covector
 $A \in \mathcal{C}^{\infty}(\M,T^*\M)$  satisfying the condition $$\delta A =0.$$
  Indeed, if $A\in H^1(\M,T^*\M)$, then we can find  a sequence $(A_k)_k$ in $\mathcal{C}^{\infty}(\M,T^*\M)$ converging towards $A$ in $H^1(\M,T^*\M).$

 Applying Lemma \ref{Lm3.2} to $A$ and $A_k$, we have the decomposition
$A= A^s+d\varphi,\,\,\delta A^s=0,\,\,\varphi_{| \p\M}=0, $ and
 $A_k= A^s_k+d\varphi_k,\,\,\delta A^s_k=0,\,\,\varphi_{k| \p\M}=0,$ for every $k \in \N$. By uniqueness of the decomposition and the estimate (\ref{3.53}), we conclude that $(A^s_k)_k$ converges to $A^s$ in $H^1(\M,T^*\M)$. By the continuity of the trace operator, we deduce the convergence in $L^2(\p_+S\M) $ of $(A^s_{k\vert\p\M})_k$ towards $(A^s_{\vert\p\M})_k$. Applying the Theorem \ref{Thm3.4} for $A^s_k$ and taking $k \rightarrow +\infty,$ we deduce that $$\Vert A^s \Vert^2_{L^2(\M,T^* \M)}
 \leq C \big( \Vert \langle \nu,(A^s) ^\sharp_{| \p\M}
 \rangle
  \Vert_{L^2(\p\M,T \M)}\Vert \I_1 (A) \Vert_{L^2(\p_+S\M)}+ \Vert \I_1 (A) \Vert^2_{H^1(\p_+S\M)}\big).$$
  \smallskip

Before starting the proof of the Theorem \ref{Thm3.4}, we need to specify some notions on tensors. For more details, one can consult \cite{[BZ]}.

Denote by $\tau_s^r\M$ the bundle of tensors of degree $(r,s)$ on $\M$.
 Let $U$ be a domain of $\M$ and denote $\mathcal{C}^{\infty}(\tau_s^r\M,U)$ the $\mathcal{C}^{\infty}(U)$-
 module of smooth sections of $\tau_s^r \M$ over $U$. We will usually be abbreviate the notation
  $\mathcal{C}^{\infty}(\tau_s^r\M,\M)$ to $\mathcal{C}^{\infty}(\tau_s^r\M).$
    Let $(x^1,\dots,x^n)$ be a local coordinate system in a domain $U$. Then  any tensor field $u\in \mathcal{C}^{\infty}(\tau_s^r \M,U)$ can be uniquely represented as
\begin{equation}\label{3.11}
u=u_{j_1,\dots,j_s}^{i_1,\dots,i_r} \frac{\p}{\p x^{i_1}}\otimes\dots\otimes\frac{\p}{\p x^{i_r}}\otimes dx^{j_1} \otimes \dots \otimes dx^{j_s}.
\end{equation}
The terms $u_{j_1,\cdots,j_s}^{i_1,\cdots,i_r} \in \mathcal{C}^{\infty}(U)$ are called the coordinates of the field $u$ in the given coordinate system. We will usually abbreviate (\ref{3.11}) on the following way
\begin{equation}\label{3.12}
 u=(u_{j_1,\cdots,j_s}^{i_1,\cdots,i_r}).
\end{equation}

We first extend the covariant differenciation defined on vector fields to tensor fields ( see \cite{[Sh]} Theorem 3.2.1 pp. 85) as follows:
\begin{equation}\label{3.14}
\nabla:\mathcal{C}^{\infty}(\tau_s^r \M)\longrightarrow \mathcal{C}^{\infty}(\tau_{s+1}^r \M)
\end{equation}
 and for a tensor field $$u=u_{j_1,\dots,j_s}^{i_1,\dots,i_r} \frac{\p}{\p x^{i_1}}\otimes\dots\otimes\frac{\p}{\p x^{i_r}}\otimes dx^{j_1} \otimes \dots \otimes dx^{j_s},$$ we define the field $\nabla u$ by $$\nabla u=\nabla_ku_{j_1,\dots,j_s}^{i_1,\dots,i_r}  \frac{\p}{\p x^{i_1}}\otimes\dots\otimes\frac{\p}{\p x^{i_r}}\otimes dx^{j_1} \otimes \dots \otimes dx^{j_s}\otimes dx^k,$$ where
\begin{equation}\label{3.15}
\nabla_ku_{j_1\dots j_s}^{i_1\dots i_r} =\frac{\p}{\p x^k}u_{j_1\dots j_s}^{i_1 \dots i_r}+\sum_{m=1}^r\Gamma^{i_m}_{kp}u_{j_1 \dots j_s}^{i_1 \dots i_{m-1} p i_{m+1} \dots i_r}-\sum_{m=1}^s\Gamma^{p}_{kj_m}u_{j_1 \dots j_{m-1} p j_{m+1} \dots j_s}^{i_1 \dots  i_r}.
\end{equation}
Next, we extend this covariant differentiation for tensors on $\M$ to tensors on $T\M$. Fix $(x^1,\dots, x^n)$ a local coordinates system in a domain $U \subset \M$, then denote by $\frac{\p}{\p x^i}$ the coordinates vector fields and by $dx^i$ the coordinates covector fields.
Let $(\xi^1,\dots,\xi^n)$ be the coordinates of a vector $\xi \in T_x\M$; that is $\xi=\xi^i \frac{\p}{\p x^i}$. Then the family of the functions $(x^1, \dots,x^n,\xi^1,\dots,\xi^n)$ is a local coordinate system \textit{associated} with $(x^1,\dots,x^n).$
In the sequel, we will only use coordinates systems on $T\M$ associated with some local coordinates systems on $\M$. In general, tensor fields defined on $T\M$ are expressed with the coordinates fields $\frac{\p}{\p x^i},\frac{\p}{\p \xi^i}, dx^i, d\xi^i.$ A tensor $u$ of degree $(r,s)$ at a point $(x,\xi)\in T\M$ is called \textit{semibasic} if in some (and so, in any) coordinates system, it can be represented by:
\begin{equation}\label{3.16}
u=u_{j_1 \dots j_s}^{i_1 \dots i_r} \frac{\p}{\p \xi^{i_1}}\otimes\dots\otimes\frac{\p}{\p \xi^{i_r}}\otimes dx^{j_1} \otimes \dots \otimes dx^{j_s},
\end{equation}
which will be abbreviated to $$u=(u_{j_1 \dots j_s}^{i_1 \dots i_r} ).$$ We denote by $\beta_s^r\M$ the subbundle of $\tau_s^r(T\M)$ containing all semibasic tensors of degree $(r,s).$ In particular  $\mathcal{C}^{\infty}(\beta_0^0 \M)=\mathcal{C}^{\infty}(T \M)$.
We will call \textit{semibasic vector fields} the elements of $\mathcal{C}^{\infty}(\beta_0^1 \M)$ and \textit{semibasic vector fields} the elements of $\mathcal{C}^{\infty}(\beta_1^0 \M)$ are called \textit{semibasic covector fields}.
We can consider tensor fields on $M$ as semibasic tensor fields on $T\M$ whose components are independent of the second argument $\xi$. Then we have the canonical embedding
\begin{equation}\label{3.17}
\iota:\mathcal{C}^{\infty}(\tau_s^r \M) \subset \mathcal{C}^{\infty}(\beta_s^r \M),
\end{equation}
 with $\iota( \frac{\p}{\p x^{i}})= \frac{\p}{\p \xi^{i}}$ and $\iota(dx^i)=dx^i.$
\smallskip
The extension of the covariant derivative to tensors of $T \M$ gives rise to two semibasic tensor fields. For $u\in \mathcal{C}^{\infty}(\beta_s^r \M)$, we define two semibasic tensor fields $\overset{\vv}{\nabla}u$ and $\overset{\hh}{\nabla}u$ by
\begin{equation}\label{3.18}
\overset{\vv}{\nabla}u =\overset{\vv}{\nabla}_k u_{j_1 \dots j_s}^{i_1 \dots i_r} \frac{\p}{\p \xi^{i_1}}\otimes\dots\otimes\frac{\p}{\p \xi^{i_r}}\otimes dx^{j_1} \otimes \dots \otimes dx^{j_s}\otimes dx^k,
\end{equation}
where
\begin{equation}\label{3.19}
\overset{\vv}{\nabla}_k u_{j_1 \dots j_s}^{i_1 \dots i_r}=\frac{\p }{\p \xi^k} u_{j_1 \dots j_s}^{i_1 \dots i_r},
\end{equation}
and
\begin{equation}\label{3.20}
\overset{\hh}{\nabla}u =\overset{\hh}{\nabla}_k u_{j_1 \dots j_s}^{i_1 \dots i_r} \frac{\p}{\p \xi^{i_1}}\otimes\dots\otimes\frac{\p}{\p \xi^{i_r}}\otimes dx^{j_1} \otimes \dots \otimes dx^{j_s}\otimes dx^k,
\end{equation}
where
\begin{multline}\label{3.21}
\overset{\hh}{\nabla}_ku_{j_1\dots j_s}^{i_1\dots i_r} =\frac{\p}{\p x^k}u_{j_1\dots j_s}^{i_1 \dots i_r}-\Gamma^{p}_{kq}\xi^q \frac{\p}{\p \xi^{p}}u_{j_1\dots j_s}^{i_1\dots i_r}\cr
+\sum_{m=1}^r\Gamma^{i_m}_{kp}u_{j_1 \dots j_s}^{i_1 \dots i_{m-1} p i_{m+1} \dots i_r}-\sum_{m=1}^s\Gamma^{p}_{kj_m}u_{j_1 \dots j_{m-1} p j_{m+1} \dots j_s}^{i_1 \dots  i_r},
\end{multline}
where $\Gamma_{kq}^p$ is the Christoffel symbol.
The differential operators $\overset{\vv}{\nabla}, \overset{\hh}{\nabla}:\mathcal{C}^{\infty}(\beta_s^r \M)\longrightarrow \mathcal{C}^{\infty}(\beta_{s+1}^r \M)$ are respectively called the \textit{vertical} and the \textit{horizontal} covariant derivatives.\\
In particular, for $u\in\mathcal{C}^\infty(T\M)$, we have
\begin{equation}\label{3.22}
\overset{\hh}{\nabla} u=(\overset{\hh}{\nabla}_k u)dx^k,\quad \overset{\hh}{\nabla}_k u=\frac{\p u}{\p x^k}-\Gamma_{kq}^p\xi^q\frac{\p u}{\p\xi^p},
\end{equation}
and
\begin{equation}\label{3.23}
\overset{\vv}{\nabla}u=(\overset{\vv}{\nabla}_ku)dx^k,\quad \overset{\vv}{\nabla}_ku=\frac{\p u}{\p\xi^k}.
\end{equation}
We have the following properties ( \cite{[Sh]}, pp. 95):
\begin{equation}\label{3.24}
\overset{\vv}{\nabla}_k \overset{\hh}{\nabla}_l=\overset{\hh}{\nabla}_l\overset{\vv}{\nabla}_k
\end{equation}
and
\begin{equation}\label{3.25}
\overset{\hh}{\nabla}_k\xi^i=0,\quad \overset{\vv}{\nabla}_k\xi^i=\delta_k^i.
\end{equation}
The well-defined differential operators $\overset{\vv}{\nabla}$ and $\overset{\hh}{\nabla}$ are of
 first-order and they are extended naturally as operators to the Sobolev space $H^1(\beta^r_s\M)$.
Now, we define the vertical divergence $\overset{\vv}{\dive}$ and the horizontal divergence $\overset{\hh}{\dive}$ of a semibasic vector field $V$ by
\begin{equation}\label{3.26}
 \overset{\vv}{\dive}(V)= \overset{\vv}{\nabla}_k v^k,\quad  \overset{\hh}{\dive}(V)= \overset{\hh}{\nabla}_k v^k.
\end{equation}
For a semibasic vector field $V$ which is homogeneous of degree $k$ in its second argument, we have the following divergence formulas ( \cite{[Sh]}, p 101). \\ For $k+n-1\neq 0$ we have
the \textit{Gauss-Ostrogradskii formula} of the vertical divergence
\begin{equation}\label{3.27}
\int_{S\M} \overset{\vv}{\dive}(V)\, \dvv=(k+n-1) \int_{S\M} \langle V,\xi\rangle \, \dvv,\quad V\in \mathcal{C}^\infty(T\M).
\end{equation}
For $k+n \neq 0$, the
\textit{Gauss-Ostrogradskii formula} of the horizontal divergence is as follows:
\begin{equation}\label{3.28}
\int_{S\M} \overset{\hh}{\dive}(V)\, \dvv= \int_{\p S\M} \langle V,\nu\rangle \, \dss, \quad V\in \mathcal{C}^\infty(T\M).
\end{equation}
Let $H$ be the vector field associated with the geodesic flow $\phi_t$ defined in (\ref{Hu}).
In coordinate form, we have
\begin{equation}\label{3.30}
H=\xi^i\frac{\p}{\p x^i}-\Gamma^i_{jk}\xi^j\xi^k\frac{\p}{\p \xi^i}=\xi^i (\frac{\p}{\p x^i}-\Gamma^p_{iq}\xi^q \frac{\p}{\p \xi^p})=\xi^i\overset{\hh}{\nabla}_i.
\end{equation}
The proof of the Theorem \ref{Thm3.4} starts by the use of the \textit{Pestov identity}:
\begin{equation}\label{3.31}
2\langle  \overset{\hh}{\nabla} u,  \overset{\vv}{\nabla}Hu \rangle=\vert  \overset{\hh}{\nabla} u \vert^2+  \overset{\hh}{\dive}(V) + \overset{\vv}{\dive}(W)- \langle R(\xi,(\overset{\vv}{\nabla} u)^{\sharp}) \xi, ( \overset{\vv}{\nabla} u)^{\sharp}\rangle, \quad u\in \mathcal{C}^\infty(T\M).
\end{equation}
Here
  the semibasic vector $V$ and $W$ are given by
 \begin{equation}\label{3.32}
  V = \langle \overset{\hh}{\nabla} u, \overset{\vv}{\nabla} u\rangle \xi - \langle \xi, ( \overset{\hh}{\nabla} u )^{\sharp} \rangle  (\overset{\vv}{\nabla} u)^{\sharp},
\end{equation}
 \begin{equation}\label{3.33}
 W=\langle \xi , (\overset{\hh}{\nabla} u)^{\sharp} \rangle  (\overset{\hh}{\nabla} u)^{\sharp},
 \end{equation}
 and $R$ is the curvature tensor.\\
  The  \textit{Pestov identity} is the basic energy identity used since
 the work of Mukhometov \cite{[Mukh]} in most injectivity proofs of ray transforms
  in absence of real-analyticity or special symmetries.
We will apply the Pestov identity to the function  $u:S\M\to\R$ defined by
   \begin{equation}\label{3.56}
    u(x,\xi)=\int_0^{\tau_+(x,\xi)}\langle A^{\sharp}(\gamma_{x,\xi}(t)),
    \dot{\gamma}_{x,\xi}(t)\rangle dt.
   \end{equation}
   Here  $\langle A^{\sharp}(\gamma_{x,\xi}(t)),
    \dot{\gamma}_{x,\xi}(t)\rangle = \sum_{j=1}^n \, a_j(\gamma_{x,\xi}(t)) \dot{\gamma}^j_{x,\xi}(t)\,  \dd t, (x,\xi)\in\p_+S\M$.
  The function $u$ satisfies the boundary conditions
  \begin{equation}\label{3.57}
  u=\I_1 (A),\quad\mbox{on }\p_+ S\M,
  \end{equation}
    and
    \begin{equation}\label{3.58}
     u=0,\quad\mbox{on }\p_- S\M
\end{equation}
since $\tau_+(x,\xi)=0$ for $(x,\xi)\in \p_- S\M$.
\begin{Lemm}\label{Lm3.6}
Let $u$ given by (\ref{3.56}). Then $u$ is smooth function on $T\M\backslash T(\p\M)$ and has the following properties:
\begin{enumerate}
\item  For $\lambda >0$, $u(x,\lambda \xi)=u(x,\xi).$
\item  $u$ satisfies the kinetic equation $Hu(x,\xi)=-\langle A^{\sharp}(x), \xi \rangle $.
\item $u$ satisfies the equation $H\overset{\vv}{\nabla} u=-A-\overset{\hh}{\nabla} u$.
\end{enumerate}
\end{Lemm}
\begin{proof}
Item (1) is immediate from the relations $\tau_+(x,\lambda \xi)=\lambda ^{-1} \tau_+(x,\xi)$, $\gamma_{x,\lambda \xi}(t)=\gamma_{x,\xi}(\lambda t)$ and $\dot{\gamma}_{x,\lambda \xi}(t)=\lambda \dot{\gamma}_{x,\xi}(\lambda t)$ for any $\lambda>0$. Then
$$
u(x,\lambda \xi)=\int_0^{\lambda^{-1}\tau_+(x,\xi)}\lambda \sum_{j=1}^n \,   a_j(\gamma_{x,\xi}(\lambda t)) \dot{\gamma}_{x,\xi}^j(\lambda  t)  dt =  u(x,\xi).
 $$
Prove item (2). Let $s \in \R$ sufficiently small, we set  $x_s=\gamma_{x,\xi}(s)$ and $\xi_s=\dot{\gamma}_{x,\xi}(s).$ Then, $\gamma_{x_s,\xi_s}(t)=\gamma_{x,\xi}(t+s)$ and $\tau_+(x_s,\xi_s)=\tau_+(x,\xi)-s.$ So,
\begin{align*}
 u(\gamma_{x,\xi}(s),\dot{\gamma}_{x,\xi}(s))  =u(x_s,\xi_s) & = \int_0^{\tau_+(x_s,\xi_s)}
 \sum_{j=1}^{n}a_j(\gamma_{x,\xi}(t+s))
    \dot{\gamma}_{x,\xi}^j(t+s) dt \cr & =\int_s^{\tau_+(x,\xi)}\sum_{j=1}^{n}a_j(\gamma_{x,\xi}(t))
    \dot{\gamma}_{x,\xi}^j(t) dt.
\end{align*}
We have $\gamma_{x,\xi}(0)=x$, $\dot{\gamma}_{x,\xi}(0)=\xi$ and $\ddot{\gamma}_{x,\xi}^{i}(0)=-\Gamma ^i_{jk}(x)\xi^j\xi^k$. Then, differentiating with respect to $s$ and taking $s=0$, we obtain that
 $$
 \frac{\p u}{\p x_i}\dot{\gamma}_{x,\xi}^i(0)+\frac{\p u}{\p \xi_i}  \ddot {\gamma}_{x,\xi}^{i}(0)=-\sum_{j=1}^{n} a_j(x) \xi^j .
 $$
 Since $\ddot{\gamma}_{x,\xi}^{i}(0)=-\Gamma ^i_{jk}(x)\xi^j\xi^k,$ then
  $$\xi^i\frac{\p u}{\p x_i}-\Gamma^i_{jk}\xi^j\xi^k\frac{\p u}{\p \xi_i}=
 -\sum_{j=1}^{n} a_j(x) \xi^j .$$
 Thus we have $Hu(x,\xi)=-\sum_{j=1}^{n} a_j(x) \xi^j =-\langle A^{\sharp}(x), \xi \rangle$.
 \smallskip

To prove item (3), we apply the operator $\overset{\vv}{\nabla}$ to the kinetic equation. We obtain $ \overset{\vv}{\nabla} (Hu)=-\overset{\vv}{\nabla} \langle A^{\sharp}(x), \xi \rangle.$ It follows that
$$
 -A= \overset{\vv}{\nabla} (Hu)  = \overset{\vv}{\nabla}_j (\xi^i  \overset{\hh}{\nabla}_i u) dx^j  =(\overset{\vv}{\nabla}_j \xi^i ) \overset{\hh}{\nabla}_i u dx^j+\xi ^i  (\overset{\vv}{\nabla}_j \overset{\hh}{\nabla}_i u) dx^j.
$$
Thus, we get
$$
-A=\overset{\hh}{\nabla} u+H(\overset{\vv}{\nabla}u).
$$

\end{proof}
In the proof of Theorem \ref{Thm3.4}, we also need the following lemma.
\begin{Lemm}\label{Lm3.7}
For  $A \in \mathcal{C}^{\infty}(\M,T^*\M)$ a semibasic covector field, the next equality is true
$$ \int_{S\M}\vert A \vert^2 \dvv = n  \int_{S\M} \langle A^{\sharp}(x),\xi \rangle ^2 \dvv.$$
\end{Lemm}
\begin{proof}
We let $x \in \M $  and we consider the map defined on $S_x\M$ by  $$\phi_x(\xi)=-\langle A^{\sharp}(x),\xi \rangle.$$ We denote by  $B_x\M$  the unit ball of $T_x\M$ then for any $\xi \in B_x \M$ with $ \xi\neq 0$, we set
   $$\widetilde {\phi}_x(\xi)=\phi_x(\frac{\xi}{\vert \xi \vert}).$$ We will apply the Green formula to $\widetilde {\phi}_x(\xi)$ on $B_x\M.$ First of all we choose a local coordinate system in some neighbourhood of $x$ such that $\g_{ij}(x)=\delta_{ij}$. Thus we can identify  $T_x\M$ with the euclidean space $\R^n$, and $S_x\M$ with the unit sphere $S^{n-1}$ of $\R^n$, and $B_x\M$ with the unit ball $B^n$  of $\R^n$. We equip $B_x\M$  with a measure $\lambda_x$ which is identified to the Borelian measure $d\lambda$ on $B^n$.
  Applying the Green formula (\ref{2.10}), we get
   \begin{equation}\label{3.42}
  \int_{B_x\M}\vert \nabla_{\xi} \widetilde {\phi}_x(\xi) \vert^2 d\lambda_x(\xi)+\int_{B_x\M}  \widetilde {\phi}_x(\xi)\Delta_{\xi} \widetilde {\phi}_x(\xi) d\lambda_x(\xi) =  \int_{S_x\M} \widetilde {\phi}_x(\xi)\langle \frac{\xi}{\vert \xi \vert }, \nabla_{\xi} \widetilde {\phi}_x(\xi)\rangle d\omega_x(\xi).
     \end{equation}
     Let us compute the integrands in the formula above. We have
      $\nabla_{\xi}(\frac{1}{\vert \xi \vert})=-\frac{\xi}{\vert \xi \vert^3}$ and $\nabla_{\xi} \phi_x= -A^{\sharp}(x)$ so $ \nabla_{\xi} \widetilde {\phi}_x(\xi)=-\frac{\xi}{\vert \xi \vert^3}\phi_x(\xi)-\frac{1}{\vert \xi \vert} A^\sharp.$
       Using the definition of $\phi_x$, we obtain $$\langle  \frac{\xi}{\vert \xi \vert}, \nabla_{\xi} \widetilde {\phi}_x(\xi)\rangle=0 $$ and the third integrand of (\ref{3.42}) vanishes.
       Then we use the relation $ \widetilde{\phi}_x= \frac{1}{\vert \xi \vert}\phi_x$ to conclude that $$\vert \nabla_{\xi} \widetilde {\phi}_x(\xi)  \vert^2= \frac{\vert A^\sharp\vert^2}{\vert \xi\vert^2}-\frac{\vert \widetilde{\phi}_x(\xi)\vert ^2}{\vert \xi \vert^2}.$$
       Writing in polar coordinates, we obtain  $$\int_{B_x\M}\vert \nabla_{\xi} \widetilde {\phi}_x(\xi) \vert^2 d\lambda_x(\xi)=\frac{1}{n-2}\int_{S_x\M} \vert A^{\sharp} \vert^2 d\omega_x(\xi)-\frac{1}{n-2}\int_{S_x\M} \vert \phi_x (\xi)\vert^2 d\omega_x(\xi).$$
   To compute the second term in ( \ref{3.42}), we apply the Leibniz formula. We get $$\Delta_{\xi} \widetilde {\phi}_x(\xi)= \Delta_{\xi}(\frac{1}{\vert \xi \vert}) \phi_x(\xi)+2 \langle \nabla_{\xi}(\frac{1}{\vert \xi \vert}),\nabla_{\xi} \phi_x(\xi)\rangle+\frac{1}{\vert \xi \vert}\Delta_{\xi} \phi_x(\xi).$$
   Since $$ \Delta_{\xi}(\frac{1}{\vert \xi \vert})=-\frac{(n-3)}{\vert \xi \vert^3},\quad  \Delta_{\xi} \phi_x=0  \quad \mbox{ and } \quad \widetilde{\phi}_x(\xi)= \frac{1}{\vert \xi \vert}\phi_x(\xi) ,$$
    we conclude that $$\Delta_{\xi} \widetilde {\phi}_x(\xi)=\frac{(1-n)}{\vert \xi \vert^2}\widetilde {\phi}_x(\xi).$$
  In polar coordinates, we get $$\int_{B_x\M}  \widetilde {\phi}_x(\xi)\Delta_{\xi} \widetilde {\phi}_x(\xi) d\lambda_x(\xi) = \frac{(1-n)}{n-2}\int_{S_x\M} \vert \phi_x (\xi)\vert^2 d\omega_x(\xi).$$
   The identity (\ref{3.42}) becomes
   \begin{equation}\label{3.58'}
  \frac{1}{n-2}\int_{S_x\M} \vert A^{\sharp} \vert^2 d\omega_x(\xi)-\frac{1}{n-2}\int_{S_x\M} \vert \phi_x (\xi)\vert^2 d\omega_x(\xi)+ \frac{(1-n)}{n-2}\int_{S_x\M} \vert \phi_x (\xi)\vert^2 d\omega_x(\xi)=0.
    \end{equation}
         This yields to $$\int_{S_x\M}\vert A \vert^2 d\omega_x(\xi)=n\int_{S_x\M} \langle A^{\sharp}(x),\xi \rangle ^2 d\omega_x(\xi).$$
     Finally, integrating the last equality with respect to $x \in \M$, we obtain $$ \int_{S\M}\vert A \vert^2 \dvv = n  \int_{S\M} \langle A^{\sharp}(x),\xi \rangle ^2 \dvv$$
     and the lemma is done.
\end{proof}
Before starting the proof of Theorem \ref{Thm3.4}, we state a last lemma proved in \cite{[Sh]}.
\begin{Lemm}\label{Lm3.8}
Let $(\M,\g)$ be a simple Riemannian manifold and let $u\in\mathcal{C}^\infty(\beta_m^0\M)$ be a semibasic tensor field  satisfying the boundary condition $u_{/ \p_-S\M}=0$, then the following estimate
\begin{equation}\label{3.37}
\int_{S\M} K^+(x,\xi) \vert u(x,\xi) \vert ^2 \dvv \leq k^+\int_{S\M} \vert Hu(x,\xi) \vert ^2 \dvv
\end{equation}
holds true.
\end{Lemm}

\subsection*{Proof of Theorem \ref{Thm3.4}}
Recall that we prove the theorem for a real covector $A \in \mathcal{C}^{\infty}(\M,T^*\M)$ satisfying the condition $\delta A =0.$ The proof consists in combining the Lemmas \ref{Lm3.6},   \ref{Lm3.7} and \ref{Lm3.8} in the Pestov identity (\ref{3.31}).
 For $u \in \mathcal{C}^{\infty}(T \M)$, we have
\begin{equation}\label{3.59}
2\langle  \overset{\hh}{\nabla} u,  \overset{\vv}{\nabla}Hu \rangle=\vert  \overset{\hh}{\nabla} u \vert^2+  \overset{\hh}{\dive}(V) + \overset{\vv}{\dive}(W)- \langle R(\xi,(\overset{\vv}{\nabla} u)^{\sharp})\xi, ( \overset{\vv}{\nabla} u)^{\sharp}\rangle,
\end{equation}
the semibasic vectors $V$ and $W$ are given by
 \begin{equation}\label{3.60}
  V = \langle \overset{\hh}{\nabla} u, \overset{\vv}{\nabla} u\rangle\xi-\langle\xi, (\overset{\hh}{\nabla} u)^{\sharp}\rangle  (\overset{\vv}{\nabla} u)^{\sharp},
\end{equation}
 \begin{equation}\label{3.61}
 W=\langle\xi,(\overset{\hh}{\nabla} u)^{\sharp}\rangle  (\overset{\hh}{\nabla} u)^{\sharp}.
 \end{equation}
Combining the Lemma \ref{Lm3.6} (2) with the condition $\overset{\hh}{\dive}A^{\sharp}=\delta A=0,$ the Pestov identity (\ref{3.59}) becomes
\begin{equation}\label{3.62}
\vert  \overset{\hh}{\nabla} u \vert^2- \langle R(\xi,(\overset{\vv}{\nabla} u)^{\sharp})\xi,  (\overset{\vv}{\nabla} u)^{\sharp}\rangle=- \overset{\hh}{\dive}(2uA^{\sharp}+V) - \overset{\vv}{\dive}(W).
\end{equation}
To avoid eventual singularities of $u$  on $T (\p\M)$, we will consider the manifold $\M_{\rho}$ defined by $$\M_{\rho}=\{x \in \M,\quad d_\g(x,\p \M) \geq \rho \},$$ where $\rho >0$.  Integrating  (\ref{3.62}) over $S\M_{\rho}$ and using the divergence formula (\ref{3.27}) and (\ref{3.28})  ($W$ is positively homogeneous of degree $1$), we find that for $n \geq 1$,
 \begin{multline*}
 \int_{S\M_{\rho}}\left[ \vert  \overset{\hh}{\nabla} u \vert^2- \langle R(\xi,(\overset{\vv}{\nabla} u)^{\sharp})\xi, ( \overset{\vv}{\nabla} u)^{\sharp}\rangle\right] \dvv=
 -\int_{S\M_{\rho}} \!\left[ \overset{\hh}{\dive}(2uA^{\sharp}+V)  + \overset{\vv}{\dive}(W)\right] \dvv \cr
 =  -\int_{\p S\M_{\rho}} \langle 2uA^{\sharp}+V, \nu \rangle  d\sigma^{2n-2} -n\int_{S\M_{\rho}} \langle W, \xi \rangle \dvv,
 \end{multline*}
 where $\nu=\nu_{\rho}(x)$ is the unit vector of the outer normal to the boundary of $\M_{\rho}$.
 In view of (\ref{3.61}), we have
 $$\langle W, \xi \rangle=\langle\xi,(\overset{\hh}{\nabla} u)^{\sharp}\rangle ^2 = \vert Hu\vert ^2.$$
Hence, we deduce the equality
\begin{equation}\label{3.63}
 \int_{S\M_{\rho}}\left[ \vert  \overset{\hh}{\nabla} u \vert^2- \langle R(\xi,(\overset{\vv}{\nabla} u)^{\sharp})\xi,  ( \overset{\vv}{\nabla} u)^{\sharp}\rangle+n \vert Hu\vert^2 \right]  \dvv= -\int_{\p S\M_{\rho}} \langle2uA^{\sharp}+ V, \nu \rangle \,d\sigma^{2n-2}.
\end{equation}
Now, we wish  to pass to the limit as $\rho \rightarrow 0.$ We will apply the Lebesgue dominated convergence theorem.  Denote by $\mathbb{\chi}_\rho$ the characteristic function of the set $S\M_{\rho}$ and by $p$  the projection $p : \p S\M \longrightarrow  \p S\M_\rho$, $p(x,\xi)=(x',\xi'),$ where $x'$
 is such that the geodesic $\gamma_{xx'}$ has length $\rho$ and intersects $\p \M$ orthogonally at $x$ and $x'$, and $\xi'$ is obtained
 by the parallel translation of the vector $\xi$ along $\gamma_{xx'}$. So the equality (\ref{3.63}) becomes
 \begin{equation}\label{3.64}
 \int_{S\M}\left[ \vert  \overset{\hh}{\nabla} u \vert^2- \langle R(\xi,(\overset{\vv}{\nabla} u)^{\sharp})\xi, ( \overset{\vv}{\nabla} u)^{\sharp}\rangle+n \vert Hu\vert ^2 \right] \mathbb{\chi}_\rho  \dvv= -\int_{\p S\M} \langle 2uA^{\sharp}+V, \nu \rangle p_* ( d\sigma^{2n-2}).
\end{equation}
Note that each integrands of (\ref{3.64}) are smooth on $S\M \setminus \p S \M$ and so, they converge towards their values almost everywhere, when $\rho \rightarrow 0$.
The functions $\vert  \overset{\hh}{\nabla} u \vert^2$ and $\vert Hu\vert ^2$  are positive. Applying Lemma \ref{Lm3.8} and then Lemma \ref{Lm3.6}, the second function satisfies
 \begin{equation}\label{3.64'}
 \int_{S\M}\vert \langle R(\xi,(\overset{\vv}{\nabla} u)^{\sharp})\xi,  (\overset{\vv}{\nabla} u)^{\sharp}\rangle\vert \mathbb{\chi}_\rho  \dvv \leq  k^+\int_{S\M}\vert \overset{\hh}{\nabla} u \vert^2 \dvv.
\end{equation}

Then we conclude that the left side of (\ref{3.64}) converges as $\rho \rightarrow 0$.  In order to apply the Lebesgue theorem in (\ref{3.64}), it remains to prove that $\vert \langle 2uA^{\sharp}+V, \nu \rangle p_*  \vert$ is bounded by a summable function on $\p S \M$ which does not depend on $\rho$. For $(x,\xi)\in \p S\M$, we denote
\begin{equation}\label{3.65}
\overset{\hh}{\nabla}_{\textrm{tan}} u =\overset{\hh}{\nabla} u -\langle\overset{\hh}{\nabla} u ,\nu\rangle\nu,\quad \overset{\vv}{\nabla}_{\textrm{tan}} u =\overset{\vv}{\nabla} u -\langle\overset{\vv}{\nabla} u,\xi\rangle\xi.
\end{equation}
We have obviously
$$
\langle\overset{\hh}{\nabla}_{\textrm{tan}} u ,\nu\rangle=\langle\overset{\vv}{\nabla}_{\textrm{tan}} u,\xi\rangle=0.
$$
Then $\overset{\hh}{\nabla}_{\textrm{tan}}$ and $\overset{\vv}{\nabla}_{\textrm{tan}}$ are in fact differential operators on $\p S\M$ and $\overset{\hh}{\nabla}_{\textrm{tan}}u$, $\overset{\vv}{\nabla}_{\textrm{tan}}u$ are completely determined by the restriction $u_{|\p S\M}$ of $u$ on $\p S\M$.\\
For $(x,\xi)\in\p S\M$, we obtain
\begin{equation}\label{3.66}
\langle 2uA^{\sharp}+ V, \nu \rangle=\langle \overset{\hh}{\nabla}_{\textrm{tan}} u , \overset{\vv}{\nabla}_{\textrm{tan}} u  \rangle\langle\xi,\nu\rangle-\langle \overset{\hh}{\nabla}_{\textrm{tan}} u , \xi\rangle\langle \overset{\vv}{\nabla}_{\textrm{tan}}  u ,\nu\rangle+ 2u\langle A^{\sharp},\nu \rangle.
\end{equation}
From (\ref{3.65}), we deduce that the derivatives $\overset{\hh}{\nabla}_{\textrm{tan}}u$ and $\overset{\vv}{\nabla}_{\textrm{tan}} u$  are locally bounded. It is important that the right-hand side of (\ref{3.66}) does not contain $\langle\overset{\hh}{\nabla} u,\nu\rangle$ and $\langle\overset{\vv}{\nabla}u,\xi \rangle.$ Taking $\rho \rightarrow 0$ in the equality (\ref{3.64}), we have
\begin{equation}\label{3.67}
\int_{S\M} \left[ \vert \overset{\hh}{\nabla} u \vert ^2-\langle R(\xi,(\overset{\vv}{\nabla} u)^{\sharp})\xi, ( \overset{\vv}{\nabla} u)^{\sharp}\rangle+
n \vert Hu\vert ^2 \right]\dvv= -\int_{\p S\M} \langle 2uA^{\sharp}+V, \nu \rangle \dss.
\end{equation}
According to (\ref{3.66}) and the boundary conditions satisfied by the function $u$, we obtain
\begin{multline*}
 \int_{\p_+ S\M} \langle 2uA^{\sharp}+V, \nu \rangle \dss =\int_{\p_+ S\M} 2(\I_1A) \langle A^{\sharp},\nu \rangle  \dss\cr
+ \int_{\p_+ S\M} \para{\langle \overset{\hh}{\nabla}_{\textrm{tan}} (\I_1 A) , \overset{\vv}{\nabla}_{\textrm{tan}} (\I_1 A)  \rangle\langle\xi,\nu\rangle-\langle \overset{\hh}{\nabla}_{\textrm{tan}} (\I_1 A) , \xi\rangle\langle \overset{\vv}{\nabla}_{\textrm{tan}} (\I_1 A) ,\nu\rangle} \dss
 \cr
 :=\int_{\p_+ S\M} 2 (\I_1A) \langle A^{\sharp},\nu \rangle  \dss +\int_{\p_+ S\M} \mathcal{Q}(\I_1 A) \dss,
 \end{multline*}
here $\mathcal{Q}u$ is a quadratic form in  $\overset{\hh}{\nabla}_{\textrm{tan}}u$ and $\overset{\vv}{\nabla}_{\textrm{tan}}u$ and hence, $\mathcal{Q}$ is a first-order  differential operator on the manifold $\p_+ S \M$. Consequently, there exists a constant $C$ such that we have
$$
\vert  \int_{\p_+ S \M} \mathcal{Q}(\I_1 A) \dss \vert \leq C \Vert \I_1 A \Vert ^2_{H^1(\p_+ S \M)},
$$
and $$
\vert \int_{\p_+ S\M} 2 (\I_1A) \langle A^{\sharp},\nu \rangle  \dss \vert \leq  C \Vert \langle A^{\sharp},\nu \rangle  \Vert_{L^2(\p\M,T \M)}\Vert \I_1 (A) \Vert_{L^2(\p_+S\M)}.
$$
We conclude that we have
\begin{equation}\label{3.68}
\vert\int_{\p S\M} \langle 2uA^{\sharp}+V, \nu \rangle \dss \vert \leq C \big( \Vert \langle (A^s)^{\sharp},\nu \rangle  \Vert_{L^2(\p\M,T \M)}\Vert \I_1 (A) \Vert_{L^2(\p_+S\M)}+\Vert \I_1 A \Vert ^2_{H^1(\p_+ S \M)} \big).
\end{equation}
Combining (\ref{3.67}) with (\ref{3.68}), we obtain that
\begin{multline}\label{3.69}
\int_{S\M} \left[ \vert \overset{\hh}{\nabla} u \vert ^2-\langle R(\xi,(\overset{\vv}{\nabla} u)^{\sharp})\xi, ( \overset{\vv}{\nabla} u)^{\sharp}\rangle+
n \vert Hu\vert ^2 \right]\dvv \cr
 \leq C \big( \Vert \langle A^{\sharp},\nu \rangle  \Vert_{L^2(\p\M,T \M)}\Vert \I_1 (A) \Vert_{L^2(\p_+S\M)}+\Vert \I_1 A \Vert ^2_{H^1(\p_+ S \M)} \big).
\end{multline}
With respect to the definitions (\ref{3.8}) and (\ref{3.9}) and combining Lemma \ref{Lm3.6} (3) and Lemma \ref{Lm3.8}, we get the estimate
\begin{equation}\label{3.70}
\vert \int_{S\M} \langle R(\xi,(\overset{\vv}{\nabla} u)^{\sharp})\xi,  (\overset{\vv}{\nabla} u)^{\sharp}\rangle  \dvv \vert \leq k^+  \int_{S\M}\vert A+ \overset{\hh}{\nabla}u\vert^2 \dvv.
\end{equation}
Since we have
\begin{equation}\label{3.71}
\vert A+ \overset{\hh}{\nabla}u\vert^2
\leq 2 (\vert A \vert ^2+\vert  \overset{\hh}{\nabla}u \vert ^2),
\end{equation}
then the Lemma \ref{Lm3.7} and Lemma  \ref{Lm3.6} (2) yield to
\begin{equation}\label{3.72}
\vert \int_{S\M} \langle R(\xi,(\overset{\vv}{\nabla} u)^{\sharp})\xi,  (\overset{\vv}{\nabla} u)^{\sharp}\rangle  \dvv \vert \leq 2nk^+  \int_{S\M}|Hu|^2 \dvv +2k^+ \int_{S\M}  \vert \overset{\hh}{\nabla} u \vert ^2 \dvv.
\end{equation}
Then the estimate (\ref{3.69}) gives
\begin{multline}\label{3.73}
 (1-2k^+)\int_{S\M}  \vert \overset{\hh}{\nabla} u \vert ^2 \dvv+n(1-2k^+) \int_{S\M} \vert Hu\vert ^2  \dvv \cr
 \leq C \big( \Vert \langle A^{\sharp},\nu \rangle  \Vert_{L^2(\p\M,T \M)}\Vert \I_1 (A) \Vert_{L^2(\p_+S\M)}+\Vert \I_1 A \Vert ^2_{H^1(\p_+ S \M)}\big).
\end{multline}
So for $k^+<\frac{1}{2}$, we get that
\begin{equation}\label{3.74}
 n \int_{S\M} \vert Hu\vert ^2  \dvv
   \leq C \big( \Vert \langle A^{\sharp},\nu \rangle  \Vert_{L^2(\p\M,T \M)}\Vert \I_1 (A) \Vert_{L^2(\p_+S\M)}+\Vert \I_1 A \Vert ^2_{H^1(\p_+ S \M)} \big).
   \end{equation}
Applying Lemma \ref{Lm3.6} (2) and Lemma \ref{Lm3.7} we obtain that  for $A \in \mathcal{C}^{\infty}(\M,T^*\M)$  satisfying the condition $\delta A =0 $, we have the desired estimate of the Theorem \ref{Thm3.4}, that is
$$\Vert A \Vert^2_{L^2(\M,T^* \M)}  \leq C \big( \Vert \langle \nu,(A ^\sharp_{| \p\M}
 \rangle \Vert_{L^2(\p\M,T \M)}\Vert \I_1 (A) \Vert_{L^2(\p_+S\M)}+ \Vert \I_1 (A) \Vert^2_{H^1(\p_+S\M)}\big).$$

\section{Study of the auxiliary inverse problem}\label{S4}
\setcounter{equation}{0}
In this section, we are going to deal with Problem 2 introduced in Section \ref{S1} which concerns the electromagnetic Schr\"odinger equation (\ref{1.13}). More precisely, we aim to show a stability estimate in recovering the solenoidal part of the pure imaginary complex covector $A$ and the electric potential $q$ appearing in (\ref{1.13}) from the DN map
$N_{A,q}$. For this purpose, we have first to construct special solutions to the equation (\ref{1.13}).
 \subsection{Geometric optics solutions} \label{S4.1}
 In the sequel of the paper, $(\M,g)$ as well as the magnetic s potentials $A_{1}$ and $A_{2}$ are  extended to a simple
manifold $\M_1^{int} \Supset \M$ . We can control the $H^1(\M_{1},T^*\M_1)$ norms of $A_1$ and $A_2$ by a constant
$M_{0}>0$. Using the fact that  $A_{1}=A_{2}$ and $q_1=q_2$ on the boundary,  their extensions outside of the manifold $\M$ can  coincide  so that $A_{1}-A_{2}=0$ and $q_1-q_2=0$ in
$\M_{1} \setminus \M$.

\noindent
In the present section we aim to construct  suitable geometrical optics solutions to (\ref{1.13}), which play a crucial role in the proof of our main results. For this purpose, let us consider a function $\psi\in\mathcal{C}^2(\M)$satisfying
\begin{equation}\label{4.1}
\abs{\nabla\psi}^2=1.
\end{equation}
On the other hand, let  $\alpha\in H^1(\R,H^2(\M))$ be a solution to
\begin{equation}\label{4.2}
\left\{
  \begin{array}{ll}
  \p_t \alpha+\seq{d\psi,d\alpha}+\displaystyle\frac{1}{2} \Delta \psi\,\alpha=0,& (x,t)\in M\times \R,
   \\
   \\
   \alpha(t,x)|_{t\leq 0}=\alpha(t,x)|_{t\geq T_{0}}=0,\quad & x\in \M.
  \end{array}
\right.
\end{equation}
Finally, we assume the existence of a function $\beta_{A}\in H^1(\R,H^2(\M)) $ that  satisfies
\begin{equation}\label{4.4}
\p_t \beta_{A}+\seq{d\psi,d\beta_{A}}-i \seq{A,d\psi}\beta_{A}=0,\qquad \forall (x,t)\in M\times\R.
\end{equation}

\noindent We move now to give the coming result that claims the existence of special solutions to the equation (\ref{1.13}) whose proof is the same as the one  given in \cite{[Bellass]} (the construction remains the same  in the case  of complex
 magnetic covector fields).
\begin{Lemm}\label{Lm4.1}
Let $(A,q)\in \mathcal{C}^1(\M)\times W^{1,\infty}(\M)$. The equation
$\mathscr{H}_{A,q}u=0$ in $Q$,
$u(x,0)=0$ in $\M,
$
admits a solution in this form
\begin{equation}\label{4.6}
u(x,t)=\alpha(x,2\lambda t)\beta_{A}(x,2\lambda t)e^{i\lambda(\psi(x)-\lambda
t)}+r_\lambda(x,t),
\end{equation}
that belongs to the following space $ \mathcal{C}^1([0,T];L^2(\M))\cap\mathcal{C}([0,T];H^2(\M)).
$ Here the correction term  $r_\lambda(x,t)$ satisfies
$$r_\lambda(x,t)=0,\quad (x,t)\,\,\mbox{on}\,\, \Sigma , \qquad\mbox{and}\,\,\,\,
r_\lambda(x,0)=0,\mbox{in }\,\,\M.$$
Moreover, there exist  a positive constant $C$ that depends only on $M$ and $T$ such that, for all $\lambda \geq T_0/2T$ we have
\begin{equation}\label{4.8}
\norm{r_\lambda(\cdot,t)}_{H^k(\M)}\leq C\lambda^{k-1}\norm{\alpha}_{*},\qquad
k=0,1,
\end{equation}
where $\norm{\alpha}_*=\norm{\alpha}_{H^1(0,T_0;H^2(\M))}.$
 \end{Lemm}

In order to solve (\ref{4.1}), (\ref{4.2}) and (\ref{4.4}), we consider $ S_{y}\M_1=\set{\theta\in T_{y}\M_1,\,\,\abs{\theta}=1}$  and  $x\in \M_1$. We denote by $(r,\theta)$ the polar coordinates of $x$ in $\M_1$ with center
$y\in \p \M_1$, $r>0$ and  $\theta\in S_{y}\M_1$,  which means that $x=\exp_{y}(r\theta)$.
Thus, we have
$$
\widetilde{\g}(r,\theta)=\dd r^2+\g_0(r,\theta),
$$
where $\g_0(r,\theta)$ denotes a smooth positive definite metric.
 Proceeding as in \cite{[Bellass]}, we construct a solution to the transport equation (\ref{4.1}) in this form
\begin{equation}\label{4.9}
\psi(x)=d_\g(x,y),
\end{equation}
where $g_0$ is the geodesic distance function to $y \in \p \M_1$.
We also construct a solution $\widetilde{\alpha}(r,\theta,t)$ to the equation
\begin{equation}\label{4.13}
\frac{\p \widetilde{\alpha}}{\p t}+\frac{\p \widetilde{\alpha}}{\p
r}+\frac{1}{4}\widetilde{\alpha}\rho^{-1}\frac{\p \rho}{\p r}=0,
\end{equation}
in the following from
\begin{equation}\label{4.14}
\widetilde{\alpha}(r,\theta,t)=\rho^{-1/4}\phi(t-r)\Psi(\theta),
\end{equation}
where $\phi\in\mathcal{C}_0^\infty(\R)$ such that $\mathrm{supp}(\phi) \subset (0,1)$, $\Psi\in H^2(S_y\M)$ and $\rho$ denotes the square of the volume in geodesic polar coordinates. It is clear  that $\widetilde{\alpha}=0$ when $t\leq 0$ and $t\geq T_0$ for $T_0>1+\mathop{\rm diam} \M_{1}$. In  geodesic polar coordinates  $\nabla\psi(x)$ is defined by $\dot{\gamma}_{y,\theta}(r)$. Thus we have,
$$
\seq{\widetilde{A}(r,y,\theta),d\psi}=\seq{\widetilde{A}^\sharp(r,y,\theta),\nabla\psi}=\widetilde{\sigma}_A(\Phi_r(y,\theta)).
$$
We denote  $\widetilde{\sigma}_A(r,y,\theta)=\sigma_A(\Phi_r(y,\theta))=\seq{\dot{\gamma}_{y,\theta}(r),A^\sharp(\gamma_{y,\theta}(r))}$.
Thus $\widetilde{\beta}_A$ solves the following equation
\begin{equation}\label{4.19}
\frac{\p \widetilde{\beta}_A}{\p t}+\frac{\p \widetilde{\beta}_A}{\p
r}-i\widetilde{\sigma}_A(r,y,\theta)\widetilde{\beta}_A=0.
\end{equation}
This means that we can take $\widetilde{\beta}_A$ as
$$
\widetilde{\beta}_A(y,r,\theta,t)=\exp\para{i\int_0^t\widetilde{\sigma}_A(y,\theta,r-s)ds}.
$$
 By a similar manner, we can construct specific solutions to the backward problem.
\begin{Lemm}\label{Lm4.2}
Let $(A,q)\in \mathcal{C}^1(\M)\times W^{1,\infty}(\M)$. The magnetic Schr\"odinger equation
$\mathscr{H}^*_{A,q}v=0\,\,\textrm{in}\,\, Q,$ $
v(x,T)=0,\,\,\textrm{in}\,\, \M,$
admits a solution in this form
\begin{equation}\label{4.20}
v(x,t)=\alpha(x,2\lambda t)\beta_{\overline{A}}(x,2\lambda t)e^{i\lambda(\psi(x)-\lambda
t)}+r_\lambda(x,t).
\end{equation}
Moreover, the correction term $r_\lambda(x,t)$ satisfies
$$r_\lambda(x,t)=0,\quad (x,t)\in \Sigma , \qquad
r_\lambda(x,T)=0,\quad x\in \M.
$$
Further, there exist $C>0$ such that, for all $\lambda \geq T_0/2T$ the following estimates hold true.
\begin{equation}\label{4.22}
\norm{r_\lambda(\cdot,t)}_{H^k(\M)}\leq C\lambda^{k-1}\norm{\alpha}_{*},\qquad
k=0,1.
\end{equation}
where $\norm{\alpha}_*=\norm{\alpha}_{H^1(0,T_0;H^2(\M))}$ and the constant $C$ depends only on $T$ and $\M$.
\end{Lemm}
\subsection{Determination of the solenoidal part of the magnetic field}\label{S4.2}
In this section we are going to use the geometrical optics solutions constructed before in order
to retrieve a stability estimate for the solenoidal part $A^{s}$ of the magnetic field $A$ from the DN map $N_{A,q}$.

\noindent Let  $A_1, A_2 \in \mathscr{A}(m_{1}, k)$ and $q_1, q_2 \in \mathscr{Q}(m_{2})$, we define
$A=A_1-A_2$ and $ q=q_{1}-q_2$.
Note that we have extended $A_1$ and $A_2$ to a $H^{1}(M_1, T^{*}M_{1})$ so that $A=0$ and   and $q=0$ on $M_{1}\setminus M$.
\subsubsection{Preliminary estimates}
\begin{Lemm}\label{Lm4.3}
 Let $\alpha_j,\beta_{A_{j}}\in H^1(\R, H^2(\M))$
satisfying (\ref{4.2}) and (\ref{4.4}) \color{black} with $A=A_j$ for $j=1,2$. There exist a positive constant $C$ depending  only on $T$ and $\M$ such that
\begin{multline}\label{4.23}
\abs{2\lambda\int_{0}^T\!\!\!\!\int_{\M}\seq{A,d\psi}(\alpha_2\overline{\alpha}_1)(x,2\lambda t)(\beta_{A_{2}}\overline{\beta_{\overline{A_1}}})(x,2\lambda t)\,\dv \, \dd t } \cr
\leq C\para{\lambda^{-1}+\lambda^3\norm{N_{A_1,q_1}-N_{A_2,q_2}}}\norm{\alpha_1}_*\norm{\alpha_2}_*
\end{multline}
holds true for any $\lambda > T_{0}/2T$. \end{Lemm}
\begin{proof}
From Lemma \ref{Lm4.1}, there exists a solution $u_{2}$ to
 \begin{equation*}
\left\{
  \begin{array}{ll}
\mathscr{H}_{A_{2},q_{2}}u_{2}=0, & \mbox{in} \,\,\,Q, \\
  \\
    u_2(\cdot,0)=0, & \mbox{in}\,\,\, \M,  \\
  \end{array}
\right.
\end{equation*}
having this form
$$u_{2}(x,t)=\alpha_{2}(x,2\lambda t)\beta_{A_2}(x,2\lambda t)e^{i\lambda(\psi(x)-\lambda
t)}+r_{2,\lambda}(x,t),$$
where  $r_{2,\lambda}$ satisfies (\ref{4.8}).  On the other hand, we define $f_{\lambda}=u_{2|\Sigma}$. Let us take  $v$ a solution to
\begin{equation*}
\left\{
  \begin{array}{ll}
 \mathscr{H}_{A_{1},q_{1}}v=0, & \mbox{in} \,\,\,Q, \\
  \\
    v(\cdot,0)=0, & \mbox{in}\,\,\, \M,  \\
   \\
    v=u_{2}:=f_{\lambda}, & \mbox{on} \,\,\,\Sigma.
  \end{array}
\right.
\end{equation*}
We set $w=v-u_{2}$. Then, $w$ solves this equation
\begin{equation}
\label{4.24}
\left\{
  \begin{array}{ll}
 \mathscr{H}_{A_{1},q_{1}}w=2i\seq{A, d u_{2}}+V_{A} u_{2}+q u_{2}, & \mbox{in} \,\,\,Q, \\
   \\
    w(\cdot,0)=\p_{t}u(\cdot,0)=0, & \mbox{in}\,\,\, \M,  \\
    \\
    w=0, & \mbox{on} \,\,\,\Sigma,
  \end{array}
\right.
\end{equation}
with  $V_{A}=i \,\delta A-\seq{A_2,A_{2}}+\seq{A_{1},A_1}$.  Lemma \ref{Lm4.2}  guarantees the existence of a geometrical optic solution $u_1$ to
\begin{equation*}
\left\{
  \begin{array}{ll}
 \mathscr{H}^{*}_{A_{1},q_{1}}u_1=0, & \mbox{in} \,\,\,Q, \\
 \\
    u_1(\cdot,T)=0, & \mbox{in}\,\,\, \M,  \\
  \end{array}
\right.
\end{equation*}
in this form
$$u_1(x,t)=\alpha_{1}(x,2\lambda t)\beta_{\overline{A_1}}(x,2\lambda t)e^{i\lambda(\psi(x)-\lambda
t)}+r_{1,\lambda}(x,t),$$
where $r_{1,\lambda}$ satisfies (\ref{4.22}). We multiply the first equation in (\ref{4.24}) by $\overline{u}_{1}$ and we integrate  by parts, we find out
\begin{eqnarray}\label{4.25}
\int_{0}^T\int_{\M}2i\seq{A,d u_{2}}\overline{u}_{1}\,\dv\,dt&=&\int_{0}^T\int_{\p\M}(N_{A_{2},q_{2}}-N_{A_{1},q_{1}})(f_{\lambda})\, \overline{u}_1\,\ds\,dt \cr
&&\qquad\qquad\qquad\qquad\qquad-\int_{0}^T\int_{\M}(V_{A}+q)u_{2}\,\overline{v}\,\dv\,dt.
\end{eqnarray}
On the other hand, by replacing $u_{2}$ and $u_1$ by their expressions, we get
\begin{multline}\label{4.26}
2\lambda\int_0^T\!\!\!\int_\M \seq{A,d\psi} (\alpha_2 \overline{\alpha}_1)(x,2\lambda t)(\beta_{A_2}\overline{\beta_{\overline{A_{1}}}})(x,2\lambda t)\dv\,\dd t =\cr
\int_0^T\!\!\!\int_{\p \M}  \para{N_{A_2,q_2}-N_{A_1,q_2}}(f_{\lambda})\,\overline{u}_1 \ds \, \dd t
-2\lambda\int_0^T\!\!\!\int_\M \seq{A,d\psi} (\alpha_2\beta_{A_2})(x,2\lambda t))\overline{v}_{1,\lambda}e^{i\lambda(\psi-\lambda t)}\dv\dd t\cr
\!\!\!\!+2i\int_0^T\!\!\!\int_\M\!\!\!\! \seq{A,d(\alpha_2\beta_{A_2})}(x,2\lambda t)(\overline{\alpha}_1\overline{\beta_{\overline{A_1}}})(x,2\lambda t)\dv\dd t\cr
+2i\int_0^T\!\!\!\int_\M\!\!\!\!\! \seq{A,d(\alpha_2\beta_{A_2})}(x,2\lambda t)\overline{v}_{1,\lambda}e^{i\lambda(\psi-\lambda t)}\dv\dd t\cr
+2i\int_0^T\!\!\!\int_\M \seq{A,d r_{2,\lambda}}(\overline{\alpha}_1\overline{\beta_{\overline{A_1}}})(x,2\lambda t) e^{-i\lambda(\psi-\lambda t)}\dv\dd t
+2i\int_0^T\!\!\!\int_\M \seq{A,d r_{2,\lambda}}\overline{v}_{1,\lambda}\dv\dd t\cr
+\int_0^T\!\!\!\int_\M (V_{A}+q)(x)u_2\overline{u}_1\dv dt
= \int_0^T\!\!\!\int_{\p \M} \para{N_{A_2,q_2}-N_{A_1,q_1}}(f_{\lambda}) \, \overline{u}_1\ds \, \dd t+\mathscr{Q}_\lambda.
\end{multline}
From (\ref{4.22}) and (\ref{4.8}), on can see that
\begin{equation}\label{4.27}
\abs{\mathscr{Q}_\lambda}\leq \frac{C}{\lambda}\norm{\alpha_1}_{*}\norm{\alpha_2}_{*}.
\end{equation}
Next, applying the trace theorem, we obtain
\begin{eqnarray*}
\bigg|\int_0^T\!\!\!\int_{\p M}\para{N_{A_1,q_1}-N_{A_2,q_2}}(f_{\lambda}) \overline{u}_1 \ds \, dt \bigg|
&\leq & \norm{N_{A_1,q_1}-N_{A_2,q_2}} \norm{f_\lambda}_{H^{2,1}(\Sigma)}\norm{u_1}_{L^2(\Sigma)}\cr
&\leq & C\lambda^3\norm{\alpha_1}_{*}\norm{\alpha_2}_{*}\norm{N_{A_1,q_1}
-N_{A_2,q_2}}.
\end{eqnarray*}
This, (\ref{4.27}) and  (\ref{4.26}) give the desired result.
\end{proof}
\noindent Our next objective is to give a proof to the coming statement. Let us first introduce this set
    $$ S_y^+\M_{1} = \big\{\theta \in S_{y}\M_{1} , \langle \nu,\theta \rangle<0 \big\}. $$
\begin{Lemm}\label{Lm4.4}
 Let $y\in\p\M_1$. There exists a positive constant $C$ such that for all $\Psi\in H^2(S_y\M_{1})$ we have
\begin{eqnarray*}\label{4.28}
\abs{\int_{S_{y}^+\M_1}( \exp(- i\, \I_1(A)(y,\theta))-1) \Psi(\theta)\, \dd \omega_y(\theta)} \leq
C\norm{N_{A_1,q_1}-N_{A_2,q_2}}^{1/4}\norm{\Psi}_{H^2(S_y\M_{1})},
\end{eqnarray*}
 for any $y\in\p\M_1$. Here $C$ depends on $T$ and $\M$.
\end{Lemm}
\begin{proof}
Let $T_{0}>1+\mathop{\rm diam} \M_{1}$. We consider two  solutions  $\widetilde{\alpha}_1$ and $\widetilde{\alpha}_2$ to (\ref{4.2}) is these forms
\begin{align*}
\widetilde{\alpha}_1(r,\theta,t)=\rho^{-1/4}\phi(t-r)\Psi(\theta), \quad\mbox{and}\quad
\widetilde{\alpha}_2(r,\theta,t)=\rho^{-1/4}\phi(t-r).
\end{align*}
We set  $x=\exp_{y}(r\theta)$ for some $r>0$
and $\theta\in S_{y}\M_1$. Then, from (\ref{4.23}) one can see that
\begin{multline}
\quad 2\lambda\int_0^T\!\!\int_\M \seq{A,d\psi} (\overline{\alpha}_1\alpha_2)(x,2\lambda t)(\overline{\beta_{\overline{A_1}}}\beta_{A_2})(x,2\lambda t)  \dv \, \dd t \cr
\qquad \quad=2\lambda\int_0^T\!\!\int_{S_{y}^+\M_1}\!\int_0^{\tau_+(y,\theta)}\widetilde{\sigma}_A(r,y,\theta)(\overline{\widetilde{\alpha}}_1\widetilde{\alpha}_2)(r,\theta,2\lambda t)(\overline{\widetilde{\beta}_{\overline{A_1}}}\widetilde{\beta}_{A_2})(r,\theta,2\lambda t)
\rho^{1/2} \, \dd r \, \dd\omega_y(\theta) \, \dd t\cr
=2\lambda\int_0^T\!\!\int_{S_{y}^+\M_1}\!\int_0^{\tau_+(y,\theta)}\widetilde{\sigma}_A(r,y,\theta)\phi^2(2\lambda t-r)(\overline{\widetilde{\beta}}_{A_{1}}\widetilde{\beta}_{A_{2}})(r,\theta,2\lambda t)\Psi(\theta) \, \dd r \,
\dd\omega_y(\theta)
\, \dd t\cr
\qquad\quad=2\lambda\int_0^T\!\!\int_{S_{y}^+\M_1}\!\int_\R \widetilde{\sigma}_A(2\lambda t-\tau,y,\theta)\phi^2(\tau)(\overline{\widetilde{\beta}_{\overline{A_1}}}\widetilde{\beta}_{A_2})(2\lambda t-\tau,\theta,2\lambda t)\Psi(\theta)\, \dd \tau \,
\dd\omega_y(\theta)
\, \dd t\cr
\qquad\qquad=2\lambda\int_0^T\!\!\int_{S_{y}^+\M_1}\!\int_\R \widetilde{\sigma}_A(2\lambda t-\tau,y,\theta)\phi^2(\tau)\exp\para{-i\int_0^{2\lambda t}\widetilde{\sigma}_A(s-\tau,y,\theta)ds}
\Psi (\theta) \, \dd \tau \,
\dd\omega_y(\theta)\cr
=-\int_\R\phi^2(\tau) \!\!\int_{S_{y}^+\M_1}\int_0^T \frac{d}{dt}\exp\para{-i\int_0^{2\lambda t}\widetilde{\sigma}_A(s-\tau,y,\theta)ds}
\Psi(\theta)  \, \dd \tau \,
\dd\omega_y(\theta)\cr
=-\int_\R\phi^2(\tau) \!\!\int_{S_{y}^+\M_1}\cro{\exp\para{-i\int_0^{2\lambda T}\widetilde{\sigma}_A(s-\tau,y,\theta)ds}-1}
\Psi(\theta)  \, \dd \tau \,
\dd\omega_y(\theta).
\end{multline}
Now, bearing in mind the properties of  $\phi$, we obtain
\begin{multline*}
  \int_\R\phi^2(\tau) \!\!\int_{S_{y}^+\M_1}\cro{\exp\para{-i\int_0^{2\lambda T}\widetilde{\sigma}_A(s-\tau,y,\theta)ds}-1}
\Psi (\theta) \, \dd \tau \,
\dd\omega_y(\theta)=\cr
\int_{S_{y}^+\M_1}\cro{\exp\para{-i\int_0^{\tau_+(\theta)}\widetilde{\sigma}_A(s,y,\theta)ds}-1}
\Psi(\theta)  \,
\dd\omega_y(\theta).
\end{multline*}
The estimate  (\ref{4.23}) with (\ref{4.28}) yields
\begin{multline}\label{4.29}
\abs{\int_{S_{y}^+\M_1}\para{\exp\para{i\,\I_1(A)(y,\theta)}-1}
\Psi(\theta)  \,
\dd\omega_y(\theta)}\cr
 \leq C\para{\lambda^{-1}+\lambda^3\norm{N_{A_1,q_1}-N_{A_2,q_2}}} \norm{\Psi}_{H^2(S_y\M_{1})}.
\end{multline}
Next we minimize compared to the parameter $\lambda$ in the previous estimate we get
\begin{equation*}
\abs{\int_{S_{y}^+\M_1}\para{\exp\para{-i\,\I_1(A)(y,\theta)}-1}
\Psi(\theta)  \,
\dd\omega_y(\theta)} \leq C\norm{N_{A_1,q_1}-N_{A_2,q_2}}^{1/4} \norm{\Psi}_{H^2(S_y\M_{1})}.
\end{equation*}
\end{proof}
\noindent We shall now introduce the Poisson kernel for  the unit ball $B(0,1)\subset T_y \M_1$ as follows:
$$
P(\xi,\theta)=\frac{1-\abs{\xi}^2}{\alpha_n\color{black}\abs{\xi-\theta}^n},\quad \xi\in B(0,1);\,\, \theta\in S_y\M_1,
$$
where $\alpha_n$ is the spherical volume. For $\rho\in(0,1)$, we introduce the function $\Psi_\rho: S_y\M_1\times S_y\M_1\to \R$ as follows:
\begin{equation}\label{4.30}
\Psi_\rho(\xi,\theta)=P(\rho\,\xi,\theta),\quad (\xi,\theta)\in S_y M_1\times S_y M_1.
\end{equation}
Let us give some properties of the considered function $\Psi_\rho$. The proof of this statement can be found in \cite{[BZ]}.
\begin{Lemm}\label{Lm4.5}Let $\Psi_\rho$ be defined by (\ref{4.30}) for $\rho\in(0,1)$. Then there exists $C>0$ such that we have
 \begin{equation} \label{(1)}
\displaystyle 0\leq\psi_\rho(\xi,\theta)\leq \frac{2}{\alpha_n(1-\rho)^{n-1}},\,\,\forall\, \rho\in (0,1),\,\,\mbox{ and}\,\,\, \forall\theta\in S_y\M_1.
\end{equation}\label{(2)}
 \begin{equation}\displaystyle \int_{S_y\M_1}\psi_\rho(\xi,\theta) d\omega_y(\theta)=1,\,\,\forall\,\rho\in (0,1) \,\,\mbox{and}\,\,\forall \xi\in S_y\M_1.
 \end{equation}
\begin{equation}\label{(3)}
\int_{S_y\M_1}\psi_\rho(\xi,\theta) \abs{\xi-\theta}d\omega_y(\theta)\leq C(1-\rho)^{1/2n},\qquad \forall\,\rho\in(0,1),\,\,\forall\,\xi\in S_y\M_1.
\end{equation}
 \begin{equation}\label{(4)}
\|\Psi_\rho(\xi,\cdot)\|^2_{H^2(S_y \M_1)}\leq \frac{C}{(1-\rho)^{n+3}},\qquad \forall\rho\in(0,1),\,\,\forall\,\xi\in S_y\M_1.
\end{equation}

\end{Lemm}
\begin{Lemm}\label{Lm4.6}
 There exist $C>0,$
$\delta>0$, $\beta>0$ and $\lambda_{0}>0$ such that  for all $(y,\xi)\in\p^+S\M_1$ we have
\begin{equation}\label{4.35}
|\I_1(A)(y,\xi)|\leq C \Big( \lambda^{\delta}\|N_{A_{2},q_{2}}-N_{A_{1},q_{1}}\|+\lambda^{-\beta}  \Big),
\end{equation}
for any $\lambda>\lambda_{0}$. Here $C$ depends only on $\M$, $T$,
$m_{1}$.
\end{Lemm}
\begin{proof}
We fix $(y,\xi)\in \p_{+}SM_1$ and we assume that  $\I_1(A)$ is zero on $\p_{-}SM$. We have
\begin{eqnarray*}
\Big|\exp\para{-i\,\I_1(A)(y,\xi)} -1 \Big|=\Big|
\displaystyle\int_{S_y\M_1}\Psi_{\rho}(\xi,\theta)\Big( \exp(-i\,\I_1(A)(y,\xi )-1  \Big)\,d\omega_y(\theta)\Big|\cr
\leq \Big| \displaystyle\int_{S_y\M_1}\Psi_{\rho}(\xi,\theta)\Big(\exp( -i\,\I_1(A)(y,\xi) )
-\exp(-i\, \I_1(A)(y,\theta))   \Big)  \,d\omega_y(\theta)\Big|\cr
+\Big| \displaystyle\int_{S_y\M_1} \Psi_{\rho}(\xi,\theta)\Big(\exp(-i\,\I_1(A)(y,\theta) )-1 \Big)
 d\omega_y(\theta)\Big|.
\end{eqnarray*}
Bearing in mind that
$$\begin{array}{lll} \Big| \!\exp(
\!-i\,\I_1(A)(y,\xi))
\!-\!\exp(\! -i\,\I_1(A)(y,\theta)
\! ) \Big|\!\!\!&\leq&\!\!\! C\Big|\displaystyle \,i\,\I_1(A)(y,\xi)-i\,\I_1(A)
(y,\theta)\Big|,\cr
\end{array}$$
and in light of
$$
\Big|i\,\I_1(A)(y,\xi)-i\,\I_1(A)(y,\theta)\Big|\leq C \,|\theta-\xi|,
$$
 one gets in view of Lemma \ref{Lm4.4} with $\Psi=\Psi_{\rho}(\xi,\cdot )$ this inequality
\begin{multline*}
\Big|\exp\Big(-i\,\I_1(A)(y,\xi)\Big)-1
\Big|\leq C\int_{S_y\M_1}\!\!\!\Psi_{\rho}(\xi,\theta)\,|\theta-\xi|\,d\omega_y(\theta)\cr
\qquad\qquad\qquad\qquad\qquad+C\Big(
\lambda^{2}\|N_{A_{2},q_{2}}-N_{A_{1},q_{1}}\|+\lambda^{-1}
\Big)\|\Psi_{\rho}(\xi,\cdot)\|_{H^2(S_y\M_1)}^{2}.
\end{multline*}
Moreover, since
$$
\|\Psi_{\rho}(\xi,\cdot)\|^2_{H^2(S_y\M_1)}\leq C (1-\rho)^{-(n+3)},\quad \textrm{and}\quad \int_{S_y\M_1}\Psi_{\rho}(\xi,\theta)|\theta-\xi|\,d\omega_y(\theta) \leq C (1-\rho)^{1/2n}.
$$
Thus, we end up getting this inequality
$$
 \Big|\exp\Big(-i\,\I_1(A)(y,\xi)\Big)-1
\Big|\leq C \,(1-\rho)^{1/2n}+C\Big(
\lambda^{2}\|N_{A_{2}},q_{2}-N_{A_{1},q_{1}}\|+\lambda^{-1}
\Big)(1-\rho)^{-(3+n)}.
 $$
 Next, we select  $(1-\rho)$ so that  $(1-\rho)^{1/2n}$ coincide with $\lambda^{-1}(1-\rho)^{-(n+3)}$. Then, there exist two positive constants $\delta$ and $\beta$
satisfying
$$\Big| \exp \Big( -i\,\I_1(A)(y,\xi) \Big)-1 \Big|\leq
C \Big(
\lambda^{\delta}\|N_{A_{2},q_{2}}-N_{A_{1},q_{1}}\|+\lambda^{-\beta}
\Big).$$
 Bearing in mind that  $|B|\leq e^{M}\,|e^{B}-1|$ for any real $B$ satisfying $|B|\leq M$, one gets
$$\Big| -i\,\I_1(A)(y,\xi) \Big|\leq C\, e^{C_{m_{1},T}}\Big| \exp\Big( -i\,\I_1(A)(y,\xi)  \Big)-1
\Big|.$$
Therefor, we obtain
$$\Big|\I_1(A)(y,\xi) \Big|\leq C \Big(
\lambda^{\delta}\|N_{A_{2},q_{2}}-N_{A_{1},q_{1}}\|+\lambda^{-\beta}
\Big).$$
 This completes the proof of the Lemma.
 \end{proof}
 \subsubsection{End of the proof of the stability estimate}
At this stage, we are ready to finish the proof of the stability estimate for the solenoidal part $A^s$ of the magnetic covector $A$.

\noindent
We need first to integrate (\ref{4.35}) over $\p^+S\M_1$ with respect to $ \mu(y,\xi) \,
\dss (y,\xi)$ and then to minimize with respect to the parameter  $\lambda$ to end up getting
\begin{equation}\label{4.36}
\|\I_1(A) \|_{L^2(\p^+S\M_1)}\leq C \|N_{A_{2},q_{2}}-N_{A_{1},q_{1}}\|^{\frac{\beta}{\beta+\delta}}.
\end{equation}
On the other hand by interpolating we get\color{black}
\begin{align}\label{4.37}
\|\I_1(A) \|^2_{H^1(\p^+S\M_1)} & \leq C \|\I_1(A) \|_{L^2(\p^+S\M_1)}\|\I_1(A) \|_{H^2(\p^+S\M_1)} \cr &  \leq C \|\I_1(A) \|_{L^2(\p^+S\M_1)}.
\end{align}
Thus, from Corollary \ref{Corol3.5}, we obtain  \begin{equation}\label{4.38}
\|A^s \|^2_{L^2(\M, T^*\M)}\leq C \|N_{A_{2},q_{2}}-N_{A_{1},q_{1}}\|^{\frac{\beta}{\beta+\delta}}.
\end{equation}
From (\ref{4.38}) and (\ref{4.37}), we obtain
\begin{equation}\label{4.39}\|A^s\|_{L^2(\M, T^*\M)}\leq C \|N_{A_{2},q_{2}}-N_{A_{1},q_{1}}\|^\mu,
\end{equation}
where $\mu=\frac{\beta}{4(\beta+\delta)}.$
Moreover, let $k'\in(n/2, k)$. By the Sobolev embedding theorem and the interpolation inequality, there exists $\alpha\in(0,1)$ such that we have
 \begin{equation}\label{4.40}
 \|A^s\|_{C^0(\M)}\leq \|A^s\|_{H^{k'}(\M)}\leq C \|A^s\|^{\alpha}_{L^2(\M)}\|A^s \|^{1-\alpha}_{H^{k}(\M)}\color{black}\leq C \|A^s \|_{L^2(\M)}\leq C \|N_{A_{2},q_{2}}-N_{A_{1},q_{1}}\|^\kappa.
 \end{equation} This completes the proof of (\ref{2.24}).
\subsection{Stable determination of the electric potential}
This section is devoted to show a stability estimate for   $q$ using the stability estimate that we have already got for  $A^s$. By applying the Hodge decomposition to $A=A_1-A_2$ we get
$$A=A^s+d\varphi.$$
We denote
$$A_1'=A_1-\frac{1}{2}\,d\varphi, \qquad A_2'=A_2+\frac{1}{2}\,d\varphi,$$
so that we have
 $$A'=A_1'-A_2'=A^s.$$
The idea is to substitute $A_j$ with $A_j'$ for $j=1,2$. Since the DN map is invariant under a gauge transformation, then we have
$
N_{A_j,q_j}=N_{A_j',q_j},\quad j=1,2.
$
Then, performing the same notations of Section \ref{S4.1} and replacing  $A_j$ by $A_j'$, $j=1,2$ we get in view of  (\ref{4.25}) this estimation
\begin{multline}\label{4.41}
-\int_0^T\!\!\!\int_\M q(x)u_2\overline{u}_1\dv \, \dd t=
\int_0^T\!\!\!\int_{\p \M}\para{N_{A'_1,q_1}-N_{A'_2,q_2}} f_{\lambda}(x,t)\overline{h}_\lambda(x,t) \ds \, \dd t\cr
+\int_0^T\!\!\!\int_\M 2i\seq{A',d u_2}\overline{u}_1(x,t)\dv\,\dd t +\int_0^T\!\!\!\int_\M V_{A'}(x)u_2\overline{u}_1\dv \, \dd t
\end{multline}
where $h_\lambda$ is given by
$$
h_\lambda(x,t)=(\alpha_1\beta_{\overline{A_{1}}})(x,2\lambda t)e^{i\lambda(\psi(x)-\lambda
t)},\quad (x,t)\in \Sigma.
$$
By substituting $u_1$ and $u_2$ with their expressions and proceeding as in the proof of Lemma \ref{Lm4.3}, we get
\begin{multline}\label{4.42}
\abs{\int_{0}^T\!\!\!\!\int_{\M}q(x)(\alpha_1\overline{\alpha}_2)(x,2\lambda t)(\beta_{A_{2}}\overline{\beta}_{\overline{A_1}})(x,2\lambda t)\,\dv \, \dd t } \cr
\leq C\para{\lambda^{-2}+\lambda\|A^s\|_{\mathcal{C}^0}+\lambda^3\norm{N_{A_1,q_1}-N_{A_2,q_2}}}
\norm{\alpha_1}_*\norm{\alpha_2}_*,
\end{multline}
for all $\lambda > T_{0}/2T$.
Next, using the same arguments developed in Section \ref{S4.2} and in view of the estimation (\ref{4.40}) we may control  the $H^{1}(\p_{+}S\M_1)$ norm of the geodesic ray transform of the electric $q$  potential as follows
$$\|\I(q)\|^{2}_{H^{1}(\p_{+}S\M_1)}\leq C\|N_{A_1,q_1}-N_{A_2,q_2}\|^{\kappa_2},$$
where $\kappa_2\in (0,1)$. At this stage we just need to apply Theorem \ref{Thm3.1} in order to achieve our goal and get the desired result, that is the following estimate
\begin{equation}\label{4.43}
\|q\|_{L^{2}(\M)}\leq C \|N_{A_2,q_2}-N_{A_1,q_1}\|^s,\qquad s\in(0,1).
\end{equation}
This completes the proof of the preliminary result Theorem \ref{Thm2.3}
\section{Proof of Theorem \ref{Thm2.1}}\label{S5}
\setcounter{equation}{0}
Now we are ready to treat our main inverse problem, that is the recovery of the real vector field $X$ appearing in (\ref{1.2}) from the knwoledge of the DN map  $\Lambda_{X}$.  Lemma \ref{Lm2.2} and  Theorem \ref{Thm2.3} will  play an important role in establishing. The proof of the main result needs also  the use of the $L^2$-weighted inequality specified for the  elliptic operator $\Delta$ (see \cite{[BIB7],[BIB10]}).  In order to formulate the Carleman estimate, let us first introduce these notations:

We set $\Gamma_{0}\subset \p \M$. We suppose that there exists ${\psi}\in \mathcal{C}^{2}(\M,\R^{n})$ ( see \cite{[BMBI]}) satisfying
$${\psi}(x)>0,\,\,\,x\in\M,\qquad \quad |\nabla{\psi}(x)|>0\,\,\, x\in{\M},\qquad $$
$${\psi}(x)=0\quad x\in \p\M\setminus \Gamma_0,\qquad    \quad \p_{\nu}{\psi}(x)\leq 0\,\,\,x\in\p\M\setminus  \Gamma_{0}.$$
 Given  $\gamma>0$,  we introduce the weight function
${\eta}(x)=e^{\gamma \,{\psi}(x)}$ for any $ x\in\M.$
\begin{proposition} (see (\cite{[BIB7],[BIB10]})
  \label{Proposition 5.1}
There exist $h_{0}>0$ and $C>0$ such that for all $h\in(0,h_0)$, the estimate
\begin{multline*}
\int_{\M}(h^{-1}\,|\nabla u(x)|^{2}+h^{-3}\,|u(x)|^{2})e^{2\eta(x)/h}\,\dv\leq C\Big(  \int_{\M}|\Delta u(x)|^{2}e^{2\eta(x)/h}\,\dv\\
+\int_{\Gamma_{0}}h^{-1}\,|\p_{\nu}u(x)|^{2}e^{2\eta(x)/h}\,\ds\Big),
\end{multline*}
holds true for all $u\in H^{2}(\M)$ satisfying $u(x)=0$ on $\p\M.$
\end{proposition}
Based on Proposition \ref{Proposition 5.1}, we show in this section  the main statement of the present  paper. For this purpose, let us consider  two vectors fields  $X_{1},\,X_{2}\in \mathscr{X}(m_1)$. We define
\begin{equation}\label{5.1}
X=X_{1}-X_{2}.
\end{equation}
Our aim is to show  that $X$ stably depends on the DN map $\Lambda_{X_{1}}-\Lambda_{X_{2}}$. In view of Lemma \ref{Lm3.2},  there exists a uniquely determined
$A^s \in H^k( \M,T^*\M)$ and $\varphi \in H^{k+1}(\M)$ such that
\begin{equation}\label{5.2}
A=A^s+d \varphi,\quad \delta A^s=0,\quad \varphi|_{\p \M}=0.
\end{equation}
Thus,
\begin{equation}\label{5.3}
X^{\flat}=-2i A^{s}-2i d \varphi= X^{'\flat}-2i d \varphi,\qquad \varphi_{|\p \M}=0.
\end{equation}
Then $-2i\varphi$ is solution to the following equation
 \begin{equation}\label{5.4}
\left\{
     \begin{array}{ll}
       \Delta \phi=\Psi:=\delta X^\flat-\delta X^{'\flat}=\delta X^\flat=\dive X, & \mbox{in}\,\,\,\M, \\
       \\
       \phi=0, & \mbox{in}\,\,\,\p \M.
     \end{array}
   \right.
\end{equation}
We set $q=q_2-q_1$. Thanks to (\ref{2.15}) one can see that
\begin{eqnarray}
\Psi &=& 2\,q+\displaystyle\frac{1}{2}\seq{X_2,X_2}-\frac{1}{2}\seq{X_1,X_1}\cr &=& 2\,q-\frac{1}{2}\seq{X^{'}, X_{2} + X_{1} }-\frac{1}{2}
\seq{ \nabla  (-2i\varphi),X_2+X_1},\end{eqnarray}
where $X'$ is the vector field associated with the covector $X^{'\flat}$ defined in (\ref{5.3}). In view of Proposition \ref{Proposition 5.1}  and using the fact that $\|X_{j}\|_{L^{\infty}(\M)}\leq m_1$,  $j=1,\,2$, we find out that
\begin{multline}\label{5.6}
\int_{\M}h^{-1}|\nabla \varphi(x)|^{2}e^{2\eta(x)/h}\,\dv \leq C\Big( \int_{\M}|\Delta \varphi(x)|^{2}e^{2\eta(x)/h}\,\dv+
\int_{\Gamma_{0}}h^{-1}|\p_{\nu}\varphi(x)|^{2}e^{2\eta(x)/h}\,\ds \Big)\cr
 \leq C \Big(\int_{\M}(|q(x)|^{2}+|X'(x)|^{2}+|\nabla\varphi(x)|^{2})
e^{2\eta(x)/h}\dv+\int_{\Gamma_{0}}h^{-1}|\p_{\nu}\varphi(x)|^{2}e^{2\eta(x)/h}\,\ds \Big).
\end{multline}
Then, by taking $h$ sufficiently small, (\ref{5.6}) immediately yields
\begin{equation*}
\int_{\M} h^{-1} |\nabla \varphi(x)|^{2} e^{2\eta/h}\dv\leq C \Big(\int_{\M}(|q(x)|^{2}+|X'(x)|^{2})e^{2\eta/h}\dv+\int_{\Gamma_{0}}h^{-1}|\p_{\nu}\varphi(x)|^{2}e^{2\eta/h}\ds\Big).
\end{equation*}
This implies that
\begin{equation}\label{5.7}
\|\nabla \varphi\|^{2}_{L^{2}(\M)}\leq C \Big(  \|q_{2}-q_{1}\|^{2}_{L^{2}(\M)}+\|X'\|^{2}_{L^{2}(\M)}+\|\p_{\nu}\phi\|_{L^{2}(\Gamma_{0})}^{2} \Big).
\end{equation}
From (\ref{5.3}) and (\ref{4.39}) it is readily seen that
\begin{equation}\label{5.8}
\|X'\|_{L^{2}(\M)}= \|X^{'\flat}\|_{L^2(\M)}\leq  C\|A^s\|_{L^{2}(\M)}\leq C \|N_{A_{1},q_{1}}-N_{A_{2},q_{2}}\|^{\kappa}.
\end{equation}
  On the other hand, since  $X=0$ on $\p\M$ then $\p_{\nu} \varphi=-i/2\seq{X',\nu}$  and we get  in view of the trace theorem
\begin{equation}\label{5.9}
\|\p_{\nu}\varphi\|_{L^{2}(\Gamma_{0})}\leq C \|X'\|_{L^{2}(\Gamma)}\leq C \|X'\|_{H^{1}(\M)}\leq \| X'\|_{L^{2}(\M)}^{1/2}\|X'\|^{1/2}_{H^{2}(\M)}\leq C \|N_{A_{1},q_{1}}-N_{A_{2},q_{2}}\|^{\kappa/2}.
\end{equation}
In view of (\ref{5.3}) and (\ref{5.7}) -- (\ref{5.9}), it is easy to see that
$$\begin{array}{lll}
\|X\|_{L^{2}(\M)}\leq \|X'\|_{L^{2}(\M)}+C\|\nabla \varphi\|_{L^{2}(\M)}
\leq C \|N_{A_{2},q_{2}}-N_{A_{1},q_{1}}\|^{\tilde{s}},
\end{array}$$
for some $\tilde{s}>0$. Finally, from Lemma \ref{Lm2.2} we can  deduce the desired result.


\end{document}